\documentclass[12pt]{amsart}
\usepackage{amscd,amsmath,bm,latexsym,amssymb,amscd}
\usepackage[mathscr]{euscript}
\usepackage[all]{xy}
\usepackage{layout}
\usepackage{pb-diagram}
\usepackage{pstricks}
\usepackage{enumerate}
\usepackage{dsfont}
\usepackage{verbatim}  %%includes comment environment

\title[Representation spaces for central extensions]{Representation spaces for central extensions and almost commuting unitary matrices}

\author[A.~Adem]{Alejandro Adem}
\address{Department of Mathematics,
University of British Columbia, Vancouver BC V6T 1Z2, Canada}
\email{adem@math.ubc.ca}

\author[M.C. Cheng]{Man Chuen Cheng}
\address{Department of Mathematics, University of British Columbia, Vancouver, BC V6T 1Z2, Canada}
\curraddr{Department of Mathematics, The Chinese University of Hong Kong, Shatin, Hong Kong}
\email{mccheng@math.cuhk.edu.hk}

\thanks{The first author was supported by NSERC.
The second author would like to thank the Pacific Institute for the Mathematical
Sciences for hosting him
during the time when this work
was completed.}

\providecommand{\res}{\mathop{\rm res}\nolimits}%


\newcommand{\C}{\mathbb{C}}

\newcommand{\bZ}{\mathbb{Z}}
\newcommand{\bN}{\mathbb{N}}
\newcommand{\bQ}{\mathbb{Q}}
\newcommand{\bC}{\mathbb{C}}
\newcommand{\bS}{\mathbb{S}}
\newcommand{\bR}{\mathbb{R}}
\newcommand{\bT}{\mathbb{T}}
\newcommand{\Tnm}{T(n,\bZ/m)}
\newcommand{\TnQZ}{T(n,\bQ/\bZ)}
\newcommand{\TnRZ}{T(n,\bR/\bZ)}

\newcommand{\Hom}{\text{Hom}}
\newcommand{\Rep}{\text{Rep}}
\newcommand{\bigslant}[2]{{\left.\raisebox{.2em}{$#1$}\middle/\raisebox{-.2em}{$#2$}\right.}}

\newtheorem{theorem}{Theorem}[section] % 1st argument is your name for it
\newtheorem{lemma}[theorem]{Lemma}     % 2nd argument is what is printed
\newtheorem{corollary}[theorem]{Corollary}
\newtheorem{proposition}[theorem]{Proposition}
\newtheorem*{theoremA}{Theorem A}
\newtheorem*{theoremB}{Theorem B}
\newtheorem{definition}[theorem]{Definition}
\newtheorem{example}[theorem]{Example}

\newtheorem{remark}[theorem]{Remark}

\DeclareMathOperator{\spn}{span}

\def\quotient#1#2{%
    \raise1ex\hbox{$#1$}\Big/\lower1ex\hbox{$#2$}%
}

\begin{document}

\begin{abstract}
Let $\Gamma$ denote a central extension of the form $1\to \bZ^r\to\Gamma\to\bZ^n\to 1$.
In this paper we enumerate and describe the structure of the connected
components of the spaces of
homomorphisms $\Hom (\Gamma, U(m))$ and the associated
moduli spaces $\Rep(\Gamma, U(m))$, where $U(m)$ is the group of
$m\times m$ unitary matrices.
%We first obtain a description of the space of
%almost commuting tuples in $U(m)$ and use this to describe the components
%of the aforementioned space of homomorphisms.
\end{abstract}

\maketitle

\section{Introduction}

The space of ordered commuting $n$--tuples in a Lie group $G$ can be analyzed using a variety
of methods from algebraic topology and representation theory (see \cite{AC}); in particular
these spaces
can be identified with $\Hom(\bZ^n, G)\subset G^n$. In this paper our
goal is to consider a more
complicated source group, namely the space of homomorphisms $\Hom(\Gamma , G)$ where
$\Gamma$ is no longer abelian, but rather a central extension of the form
$1\to \bZ^r \to \Gamma \to \bZ^n\to 1$. A key ingredient we will use is the notion of spaces of
almost commuting elements (see \cite{BFM} and \cite{ACGII}). We focus our attention on
the unitary groups, for which we obtain complete descriptions. These in turn are used to
shed light on the structure of $\Hom(\Gamma, U(m))$ and the associated spaces of
representations $\Rep (\Gamma, U(m))$. An important motivation for this is the fact that
they arise as moduli spaces of isomorphism classes of flat connections on principal $U(m)$--bundles over compact manifolds $M$ which can be described as $r$--torus bundles over the $n$--torus.

%\medskip

Our results are rather intricate, as they expose a very rich structure encoding the components
of these spaces of representations. For clarity of exposition we will focus here on the case
when $r=1$; the more cumbersome general
case is described in Section 6. Let $B_n(U(m))$ denote the space of almost commuting
$n$--tuples in $U(m)$ i.e. the space of ordered
$n$--tuples $(A_1,\dots ,A_n)$ such that the pairwise commutators
$[A_i, A_j]$ are all central in $U(m)$. The characteristic polynomial defines a map
$\chi : U(m)\to \bC[z]$; for a central extension $1\to \bZ\to \Gamma\to \bZ^n\to 1$ we have a
natural restriction $\Hom(\Gamma, U(m))\to U(m)$. These two maps can be composed to yield
a function $\Hom(\Gamma, U(m))\to \bC [z]$. Given a polynomial $p(z)$ we denote its inverse
image in $U(m)$ by $U(m)_{p(z)}$ and its inverse image in $\Hom(\Gamma, U(m))$ by
$\Hom(\Gamma, U(m))_{p(z)}$. The extension $\Gamma$ is defined by a k--invariant
$\omega \in H^2(\bZ^n, \bZ)$; for our purposes we write it in the following form
(see Proposition 4.1): there exists a basis $e_1,\dots ,e_n$ of $\bZ^n$ and an integer $t\le n/2$
such that $\omega = c_1e_1^*\wedge e_{t+1}^* +\dots+c_te_t^*\wedge e_{2t}^*$
where $c_1,\dots , c_t$ are positive integers such that $c_i$ divides $c_{i+1}$
for $i=1,\dots ,t-1$. We now define $\bC [z]^m_{\Gamma}\subset \bC [z]$ as the set
of degree $m$ complex polynomials $p(z)$ such that (1) all roots of $p(z)$ are roots
of unity; and (2) if a root $\lambda$ of $p(z)$ is a primitive $k$--th root of unity, then the
multiplicity of $\lambda$ in $p(z)$ is divisible by $\mu_k = \prod_{i=1}^t k/(k,c_i)$, where
$(k,c_i)$ is the greatest common divisor of $k$ and $c_i$. We now state our main theorems,
which summarize the results in \S 3 and \S 4.

\begin{theoremA}
Let $1\to \bZ \to \Gamma\to \bZ^n\to 1$ with non--trivial k--invariant $\omega$. Then
there is a decomposition into connected components
$\Hom(\Gamma , U(m)) = \coprod_{p(z)\in \bC[z]^m_{\Gamma}} \Hom(\Gamma , U(m))_{p(z)}$,
where the number of components is given by the coefficient of $x^m$ in the generating
function $\prod_{k\ge 1} {1\over{(1- x^{\mu_k})^{\Phi (k)}}}$, where $\Phi$ is Euler's phi function.
\end{theoremA}
\noindent For example, if $\Gamma_1$ denotes the integral Heisenberg group, then $\Hom(\Gamma_1, U(m))$ has
$1,2,4,7,13$ components for $m=1,2,3,4,5$ respectively (see Example \ref{Heisenberg}). The number of
components can be explicitly determined for any $m$ using Theorem A.

%\medskip

Next we describe the structure of the components. As explained in Section 4, there is a map
$B_n(U(m))\to T(n, \bQ/\bZ)$ defined using commutators, where $T(n, \bQ /\bZ)$ denotes the set of all $n\times n$
skew--symmetric matrices with entries in $\bQ /\bZ$. These matrices can be used to count
the components of the space of $n\times n$ almost commuting matrices. Given $D\in T(n,\bQ /\bZ)$ we
let $B_n(U(M))_D$ denote its inverse image under the map above.
For $2t\leq n$ and $d_1,d_2,\ldots,d_t\neq 0\in\bQ/\bZ$, let $D_n(d_1,d_2,\ldots d_t)=(d_{ij})\in \TnQZ$ be the skew-symmetric matrix with
\[
d_{ij}=
\begin{cases}
d_k &\mbox{if } (i,j)=(k+t,k),1\leq k \leq t;\\
-d_k &\mbox{if } (i,j)=(k,k+t),1\leq k \leq t;\\
0 &\mbox{otherwise.}
\end{cases}
\]
We show that $B_n(U(m))_D$ is non-empty if and only if $m$ is divisible by $\sigma(D):=\prod{|d_k|}$. If $m=l\sigma(D)$ for some positive integer $l$, then there is a map
\begin{equation*}
\phi_D:\bigslant {\left[\bigslant{\left(U(m)/\bT^l\right)\times (\bT^n)^l}{(\prod_{i=1}^t\bZ/|d_k|)^l}\right]}{\Sigma_l}\to B_n(U(m))_D
\end{equation*}
which is a rational homology equivalence for $l\geq 1$ and is a homeomorphism for $l=1$.
Moreover, $\phi_D$ induces a homeomorphism
$\bar{B}_n(U(m))_D\cong (\bT^n)^l/\Sigma_l$
after passing to quotients by the action of $U(m)$.

\begin{theoremB}
Let $p(z)= \prod_{j=1}^s (z-\lambda_j)^{m_j}$, where $\lambda_1, \dots , \lambda_s \in \bC$ are
distinct roots which are primitive $k_j$--th roots of unity.
If $\Gamma$ is a central extension of $\bZ^n$ by $\bZ$ with k--invariant
$\omega = c_1e_1^*\wedge e_{t+1}^* +\dots+c_te_t^*\wedge e_{2t}^*$
where $c_1,\dots , c_t$ are positive integers such that $c_i$ divides $c_{i+1}$,
then for every non--empty component there is a $U(m)$--equivariant homeomorphism
\[\Hom(\Gamma , U(m))_{p(z)}\cong U(m)\times_{\prod_{j=1}^s U(m_j)} \prod_{j=1}^s
B_n(U(m_j))_{D_n(-c_1q_j,\dots,-c_tq_j)}\]
where $q_j = {1\over{2\pi\sqrt{-1}}} ~log \lambda_j$.  %Moreover the orbit space under the action of $U(m)$ is homeomorphic to a product of symmetric products of circles $\prod_{j=1}^s \bT^{l_j}/\Sigma_{l_j}$ where $l_j = \prod_{i=1}^t k_j/(k_j, c_i)$ for $j=1, \dots, s$.
Moreover the orbit space under the action of $U(m)$ is homeomorphic to a product of symmetric products of tori $\prod_{j=1}^s (\bT^n)^{l_j}/\Sigma_{l_j}$ where $l_j = m_j/(\prod_{i=1}^t k_j/(k_j, c_i))$ for $j=1, \dots, s$.
\end{theoremB}

\noindent These results give complete descriptions of the moduli spaces $\Rep(\Gamma, U(m))$,
extending the techniques and results in \cite{ACG}, \cite{ACGII}. This paper was motivated
by the results obtained in \cite{H} for the case $G=SU(2)$. As $\Gamma$ is nilpotent, by \cite{Bergeron} there are
homotopy equivalences $\Hom(\Gamma, U(m))\simeq \Hom(\Gamma, \text{GL}(n,\C))$ and
$\Rep(\Gamma, U(m))\simeq \Rep(\Gamma, \text{GL}(n,\C))$, and thus our results
also provide descriptions for these a priori more complicated spaces.

%\medskip

This paper is organized as follows: in \S 2 we provide preliminaries and background; in
\S 3 we discuss the spaces of almost commuting elements in the unitary groups;
in \S 4 and \S 5 we analyze the spaces $\Hom (\Gamma, U(m))$ where $\Gamma$ is
a central extension $1\to\bZ^r\to\Gamma\to\bZ^n\to1$.

\section{Preliminaries and background}

Let $X_1,\ldots, X_n\in U(m)$ be commuting unitary matrices. We say that $\lambda=(\lambda_1,\ldots,\lambda_n)\in \bT^n:=(\bS^1)^n$ is an $n$-tuple of eigenvalues of $(X_1,\ldots,X_n)$ if there exists a non-zero vector $v$ such that $X_iv=\lambda_iv$ for all $1\leq I \leq n$. Let
$E_{\lambda}=\bigcap_{i=1}^n E_{\lambda_i}(X_i)$,
where $E_{\lambda_i}(X_i)$ denotes the eigenspace of $X_i$ associated to the eigenvalue $\lambda_i$. It is well known that we can simultaneously diagonalize all of the matrices $X_i$. Thus there is a direct sum decomposition
\begin{equation}\label{eigendecomp}
 \bC^m=E_{\lambda^1}\oplus\ldots\oplus E_{\lambda^s}
\end{equation}
where $\lambda^1,\ldots,\lambda^s\in \bT^n$ are the distinct $n$-tuples of eigenvalues of $(X_1,\ldots,X_n)$. The decomposition is unique up to the order of the eigenvalues.

The space of ordered commuting $n$-tuples of $m\times m$ unitary matrices can be identified with $\Hom(\bZ^n,U(m))$. Let $T=\bT^m$ be the maximal torus of diagonal matrices in $U(m)$. Then $Z:=\Hom(\bZ^n,U(m))^T\cong (\bT^n)^m$ is the subspace of $\Hom(\bZ^n,U(m))$ consisting of ordered $n$-tuples of diagonal unitary matrices. Let $U(m)$ act on $U(m)\times Z$ by left multiplication on the first factor and on $\Hom(\bZ^n,U(m))$ by conjugation. Consider the $U(m)$-equvariant map
\begin{equation}\label{UmZabelian}
U(m)\times Z\to \Hom(\bZ^n,U(m))
\end{equation}
given by the conjugation action $(M,(X_i))\mapsto (MX_iM^{-1})$. Let $N=N_{U(m)}(\bT^m)$ be the normalizer and $W=N/\bT^m\cong \Sigma_m$ be the Weyl group. The map (\ref{UmZabelian}) factors through $U(m)\times_NZ\cong (U(m)/\bT^m)\times_{\Sigma_m}(\bT^n)^m$ and descends to
\begin{equation}\label{phiabelian}
\phi:(U(m)/\bT^m)\times_{\Sigma_m}(\bT^n)^m\to \Hom(\bZ^n,U(m)).
\end{equation}
The map $\phi$ has a geometric interpretation. Note that $U(m)/\bT^m$ is the space of ordered $m$-tuples of pairwisely orthogonal complex lines $(L_1,\ldots,L_m)$ in $\bC^m$. Hence, each element of the domain of $\phi$ can be regarded as an unordered $m$-tuple $[(L_1,\alpha^1),\ldots,(L_m,\alpha^m)]$ with $\alpha^1,\ldots,\alpha^m\in\bT^n$. The map $\phi$ sends such an element to the almost commuting tuple $(X_1,X_2,\ldots,X_n) \in B(U(m))_D$ such that each $\alpha^j$ is an $n$-tuple of eigenvalues of the matrices $X_i$ and each complex line $L_j$ lies in the common eigenspaces $E_{\alpha^j}$.

The map $\phi$ is surjective since commuting unitary matrices can be simultaneously diagonalized. It is not injective in general, but for $(X_1,\ldots,X_n)\in\Hom(\bZ^n,U(m))$ with eigenspace decomposition (\ref{eigendecomp}) and $\dim E_{\lambda^j}=m_j$, the preimage $\phi^{-1}(X_1,\ldots,X_n) \cong \prod_{j=1}^s(U(m_j)/\bT^{m_j})/\Sigma_{m_j}$ and is $\bQ$-acyclic \cite{B}.
The map $\phi$ is a special case of the action map
\begin{equation}\label{phigeneral}
G/T\times_WX^T\to X
\end{equation}
described in \cite{B} and \cite{AG2}, where $G$ is a connected compact Lie group with a maximal torus $T$ acting on a space $X$ with maximal rank isotropy subgroups and $W$ is the Weyl group associated to $T$. As observed in \cite{B}, for any field $\mathbb{F}$ with characteristic relative prime to $|W|$, the preimage of any point in $X$ under the map (\ref{phigeneral}) is $\mathbb{F}$-acyclic. By the Vietoris-Begle mapping theorem, it follows that
$H^{\ast}(X;\mathbb{F})\cong H^{\ast}(G/T\times_WX^T;\mathbb{F})  \cong H^{\ast}(G/T\times X^T;\mathbb{F})^W$.
In particular, $\phi$ induces an isomorphism in rational cohomology. Passing to the $U(m)$-quotients, $\phi$ becomes a homeomorphism and hence
$\Rep (\bZ^n, U(m))\cong (\bT^n)^m/\Sigma_m=\text{Sym}_m\bT^n$.
In the following sections, we will generalize $\phi$ and the results above to almost commuting $n$--tuples of unitary matrices.

\section{The space of almost commuting unitary matrices}

Let $G$ be a Lie group and $K$ be a closed subgroup contained in the center $Z(G)$ of $G$. The space of $K$-almost commuting $n$-tuples in $G$, denoted by $B_n(G,K)$, was studied in \cite{ACGII}. It consists of all ordered $n$-tuples $(A_1,A_2,\ldots,A_n)\in G^n$ such that the commutators $[A_i,A_j]\in K$ for all $1\leq i,j\leq n$.

An equivalent formulation for $K$-almost commuting $n$-tuples is in terms of group homomorphisms. Let $F_n$ be the free group on $n$ generators $a_1,\ldots,a_n$. A homomorphism $f:(F_n,[F_n,F_n])\to (G,K)$ is a group homomorphism $f:F_n\to G$ whose image $f([F_n,F_n])$ of the commutator subgroup $[F_n,F_n]\subset F_n$ is contained in $K$. It is clear that there is a bijection $f\mapsto (f(a_1),\ldots,f(a_n))$ between the sets of such homomorphisms and $K$-almost commuting $n$-tuples. We sometimes write $f\in B_n(G,K)$ to represent the corresponding almost commuting $n$-tuple.

In this section, we will study almost commuting tuples of unitary matrices. For notational simplicity, $B_n(U(m),Z(U(m)))$ will be abbreviated as $B_n(U(m))$.
\begin{lemma}\label{ACeigen}
Let $A,B$ be $m\times m$ unitary matrices with $[A,B]=\gamma I_m$. Then $\gamma^m=1$.
\end{lemma}
\begin{proof}
$\gamma^m=\det{(\gamma I_m)}=\det{[A,B]}=\det{(ABA^{-1}B^{-1})}=1$.
\end{proof}
Suppose that $f:F_n\to U(m)$ is in $B_n(U(m))$. For any $u,v\in F_n$, $[f(u),f(v)]=\gamma I_m$ for some $m$-th root of unity $\gamma$ by Lemma \ref{ACeigen}. The exponential function $z\mapsto e^{2\pi\sqrt{-1}z}$ establishes a group isomorphism between $\bR/\bZ$ and $\bS^1\subset \bC$ with inverse $w\mapsto \frac{1}{2\pi\sqrt{-1}}\log w$. The multiplicative group of $m$-th roots of unity and all roots of unity corresponds to the subgroup $\bZ[\frac{1}{m}]/\bZ\cong \bZ/m$ and $\bQ/\bZ$ of $\bR/\bZ$ respectively under this isomorphism. Hence there is a map $F_n\times F_n \to  \bZ/m\subset \bQ/\bZ$ defined by $(u,v)\mapsto \frac{1}{2\pi\sqrt{-1}}\log \gamma$. Since $f([F_n,F_n])\subset Z(U(m))$, the map factors through the abelianization of $F_n\times F_n$ and thus gives rise to a $(\bQ/\bZ)$-valued skew-symmetric bilinear form $\omega_f:\bZ^n\times \bZ^n\to \bQ/\bZ$.

For an abelian group $A$, let $T(n,A)$ be the set of all $n\times n$ skew-symmetric matrices with entries in $A$. Define a map
\begin{equation}\label{def:phi}
\rho:B_n(U(m))\to  T(n,\bQ/\bZ)
\end{equation}
by $\rho(f)=(d_{ij})$, where $[f(a_i),f(a_j)]=e^{2d_{ij}\pi \sqrt{-1}}I_m$. For $D\in\TnQZ$, let $B_n(U(m))_D=\rho^{-1}(D)$. For $f\in B_n(U(m))_D$, the ordered $n$-tuple $(f(a_1),\ldots,f(a_n))$ is said to be $D$-commuting. Note that $\rho(f)$ is the skew-symmetric matrix associated to the bilinear form $\omega_f$.
The $\bZ$-module structure on $\bQ/\bZ$ induces a bi-module structure on $\TnQZ$ over the ring of $n\times n$ matrices with integral entries. The following proposition is about the effect of change of basis of $\bZ^n$ on $\rho$.
\begin{proposition}\label{congDhomeo}
Suppose that $D,D'\in T(n,\bQ/\bZ)$ and $D'=A^TDA$ for some $A\in GL(n,\bZ)$. Then there is an automorphism $\alpha:F_n\to F_n$ such that $f\in B_n(U(m))_D$ if and only if $f\circ \alpha\in B_n(U(m))_{D'}$. Hence $\alpha$ induces a homeomorphism $\alpha^{\ast}:B_n(U(m))_D\cong B_n(U(m))_{D'}$.
\end{proposition}
\begin{proof}
Let $F_n$ be the free product on generators $a_1,\ldots,a_n$. Since every $A\in GL(n,\bZ)$ can be written as a product of finitely many elementary matrices, it suffices to prove the proposition when $A$ is an elementary matrix. If $A$ is the elementary matrix obtained from $I_n$ by adding $k$ times the $i$-th column to the $j$-th one, then $\alpha$ can be taken as the automorphism with $\alpha(a_j)=a_i^ka_j$ and $\alpha(a_l)=a_l$ for $l\neq j$. For the cases where $A$ is obtained from $I_n$ by swapping two columns or by multiplying a column by $-1$, $\alpha$ can also be chosen in the obvious way.
\end{proof}

For $0\leq t\leq n/2$ and $d_1,d_2,\ldots,d_t\neq 0\in\bQ/\bZ$, let $D_n(d_1,d_2,\ldots d_t)=(d_{ij})\in \TnQZ$ be the skew-symmetric matrix with
\[
d_{ij}=
\begin{cases}
d_k &\mbox{if } (i,j)=(k+t,k),1\leq k \leq t;\\
-d_k &\mbox{if } (i,j)=(k,k+t),1\leq k \leq t;\\
0 &\mbox{otherwise.}
\end{cases}
\]

Let $D=(d_{ij})\in\TnQZ$. By taking a common denominator, all the $d_{ij}$ are contained in the subgroup $\langle [\frac{1}{q}]\rangle\cong\bZ/q\subset \bQ/\bZ$ for some large enough integer $q$. By \cite[Proposition 4.1]{RV}, $D$ is congruent to some matrix $D_n(d_1,\ldots, d_t)$ with the orders of $d_i\in \bZ/q\subset \bQ/\bZ$ satisfying $|d_{i+1}|$ divides $|d_i|$ for all $1\leq i\leq t-1$. Hence, by Proposition \ref{congDhomeo}, understanding $B_n(U(m))_{D_n(d_1,\ldots, d_t)}$  is fundamental to the study of $B_n(U(m))$.

For $D=D_n(d_1,\ldots, d_t)$, define
\begin{equation}\label{def1sigma}
\sigma(D)=\prod |d_i|.
\end{equation}

\begin{theorem}\label{specAC}
Let $2t\leq n$, $D=D_n(d_1,\ldots, d_t)\in \TnQZ$ and $\gamma_{i}=e^{2d_i\pi \sqrt{-1}}$ for $1\leq i \leq t$. Then $B_n(U(m))_D$ is non-empty if and only if $\sigma(D)$ divides $m$. In that case, suppose $l=m/\sigma(D)\in\bN$ and $(A_1,\ldots A_n)\in B_n(U(m))_D$. Then there exist orthonormal vectors $v_1,\ldots, v_l\in \bC^m$ and $\alpha_{i'j},\beta_{ij}\in\bS^1\subset\bC$ for $1\leq i\leq t< i'\leq n$ and $1\leq j \leq l$ such that
\begin{enumerate}
\item $\{A_1^{p_1}A_2^{p_2}\ldots A_t^{p_t}v_j|\,0\leq p_i < |d_i| \text{ for } 1\leq i\leq t, 1\leq j\leq l\}$ is an orthonormal basis of $\bC^m$;
\item $A_1^{p_1}A_2^{p_2}\ldots A_t^{p_t}v_j$ is an eigenvector of $A_i$ with eigenvalue $\gamma_{i-t}^{p_{i-t}}\alpha_{ij}$ for $t+1\leq i \leq 2t$ and an eigenvector of $A_i$ with eigenvalue $\alpha_{ij}$ for $2t+1 \leq i \leq n$;
\item $A_i^{|d_i|}v_j=\beta_{ij}v_j$ for $1\leq i \leq t$.
\end{enumerate}
\end{theorem}
\begin{remark}
\begin{enumerate}
\item The theorem reduces to the well-known result of commuting unitary matrices when $t=0$. In that case, $\sigma(D)=1$ and $l=m$.
\item The characteristic polynomial $\chi_i(z)$ of $A_i$ can be expressed in terms of $\alpha_{ij}$ and $\beta_{ij}$ as follows.
 \[
	\chi_i(z)=\begin{cases}
 \prod_j(z^{|d_i|}-\beta_{ij})^{m/(|d_i|l)} &\mbox{if } 1\leq i\leq t;\\
 \prod_j(z^{|d_{i-t}|}-\alpha_{ij}^{|d_{i-t}|})^{m/(|d_{i-t}|l)} &\mbox{if } t+1\leq i\leq 2t;\\
 \prod_j(z-\alpha_{ij})^{m/l} &\mbox{if } 2t+1\leq i\leq n.
   \end{cases}
  \]
\end{enumerate}
\end{remark}
\begin{proof}
First, suppose that $B_n(U(m))_D$ is non-empty and $(A_1,\ldots A_n)\in B_n(U(m))_D$. Note that $A_{t+1},\ldots A_n$ are pairwisely commuting unitary matrices and thus can be simultaneous diagonalized. In particular, there exist eigenvalues $\alpha_{i1}\in\bS^1$ of $A_i$, $t+1\leq i\leq n$, such that the intersection $W\subset \bC^m$ of the corresponding eigenspaces is non-zero. Suppose $v\in W$. Then for any $t+1\leq i \leq 2r$,
\begin{align*}
A_iA_1^{p_1}A_2^{p_2}\ldots A_t^{p_t}v&=\gamma_{i-r}^{p_{i-r}}A_1^{p_1}A_2^{p_2}\ldots A_t^{p_t}A_iv
=\gamma_{i-t}^{p_{i-t}}A_1^{p_1}A_2^{p_2}\ldots A_t^{p_t}\alpha_{i1}v\\
&=\gamma_{i-t}^{p_{i-t}}\alpha_{i1}A_1^{p_1}A_2^{p_2}\ldots A_t^{p_t}v.
\end{align*}
Similarly, for any $2t+1\leq i \leq n$,
\begin{align*}
A_iA_1^{p_1}A_2^{p_2}\ldots A_t^{p_t}v&=A_1^{p_1}A_2^{p_2}\ldots A_t^{p_r}A_iv
=A_1^{p_1}A_2^{p_2}\ldots A_t^{p_r}\alpha_{i1}v\\
&=\alpha_{i1}A_1^{p_1}A_2^{p_2}\ldots A_t^{p_t}v.
\end{align*}
%$$A_iA_1^{p_1}A_2^{p_2}\ldots A_t^{p_t}v=A_1^{p_1}A_2^{p_2}\ldots A_t^{p_r}A_tv=A_1^{p_1}A_2^{p_2}\ldots A_t^{p_r}\alpha_{i1}v\alpha_{i1}A_1^{p_1}A_2^{p_2}\ldots A_t^{p_t}v.$$
It follows that $A_1^{|d_1|},A_2^{|d_2|},\ldots,A_t^{|d_t|}$ restrict to unitary automorphisms on $W$. Since these $A_i^{|d_i|}$ commute and are unitary, there exists an unit vector $v_1\in W$ which is a common eigenvector for $A_1^{|d_1|},A_2^{|d_2|},\ldots,A_t^{|d_t|}$. Let $A_i^{|d_i|}v_1=\beta_{i1}v_1$. Then $v_1$ satisfies the properties (2) and (3) in the theorem. Let
\[V_1=\spn\{A_1^{p_1}A_2^{p_2}\ldots A_t^{p_t}v_1|\,0\leq p_i < |d_i| \text{ for } 1\leq i\leq t\}.\]
Any two vectors in $\{A_1^{p_1}A_2^{p_2}\ldots A_t^{p_t}v_1|\,0\leq p_i < |d_i| \text{ for }  1\leq i\leq t\}$ are eigenvectors of different eigenvalues of $A_i$ for some $t+1 \leq i \leq n$. Hence the spanning set is a basis of $V_1$ and $\dim(V_1)=\prod{|d_i|}=\sigma(D)$.
Let $V_1^{\perp}$ be the orthogonal complement of $V_1$ in $\bC^m$. Then $\dim(V_1^{\perp})=m-\sigma(D)$. For any $1\leq i\leq n$, since $V_1$ is invariant under the unitary matrix $A_i$, so is $V_1^{\perp}$. By repeating the argument above to the restrictions of $A_1,\ldots,A_n$ on $V_1^{\perp}$, we conclude by induction that $\sigma(D)$ divides $m$ and there exist vectors $v_1,v_2,\ldots, v_l$ with the desired properties.

If $m$ is divisible by $\sigma(D)$, one can choose arbitrary complex numbers $\alpha_{ij},\beta_{ij}\in\bS^1$ and vectors $A_1^{p_1}A_2^{p_2}\ldots A_t^{p_t}v_j$ for $0\leq p_i < |d_i|, 1\leq j\leq l,$ which form an orthonormal basis of $\bC^n$. Any such choice uniquely determines an ordered $n$-tuple $(A_1,\ldots,A_n)\in B_n(U(m))_D$ satisfying the properties stated in the theorem. This shows that $B_n(U(m))_D$ is non-empty if $m$ is divisible by $\sigma(D)$.
\end{proof}

Let $D=D_n(d_1,\ldots,d_t)$ and $m=l\sigma(D)$ for some positive integer $l$. By Proposition \ref{specAC}, $B_n(U(m))_D$ is non-empty. We will describe a subspace $Z_D\subset B_n(U(m))_D$ which is homeomorphic to a torus. In light of Theorem \ref{specAC}, it is more convenient for us to index the rows and columns of a $m\times m$ matrix by the indexing set
  \begin{equation}\label{eqn:I}
		I=\{(p_1,p_2,\ldots,p_t,j)|\,0\leq p_i < |d_i|,1\leq j\leq l\}.
	\end{equation}
\begin{definition}%\label{def:Z}
Let $D=D_n(d_1,\ldots,d_t)\in \TnQZ$ and $\gamma_{i}=e^{2d_i\pi \sqrt{-1}}$ for $1\leq i \leq t$. Suppose that $m=l\sigma(D)$ for some positive integer $l$. Index the rows and columns of $m\times m$ matrices by the index set $I$ in (\ref{eqn:I}). Define $Z_D\subset B_n(U(m))_D$ be the subset consisting of all ordered $n$-tuples $(A_1,\ldots,A_n)$ of the following forms, parametrized by $\alpha_{i'j},\beta_{ij}\in\bS^1\subset\bC$ for $1\leq i\leq t< i'\leq n$ and $1\leq j \leq l$:
\begin{enumerate}
\item
For $1\leq i\leq t,$
\[
(A_i)_{\mu\nu}=
\begin{cases}
1 &\mbox{if } \mu=(p_1,\ldots,p_i,\ldots,p_t,j),\\
&\nu=(p_1,\ldots,p_i-1,\ldots,p_t,j); \\
\beta_{ij} &\mbox{if } \mu=(p_1,\ldots,p_{i-1},1,p_{i+1},\ldots,p_t,j),\\
& \nu=(p_1,\ldots,p_{i-1},|d_i|-1,p_{i+1},\ldots,p_t,j);\\
0 &\mbox{otherwise.}\\
\end{cases}
\]

\item
For $t+1\leq i\leq 2t$,
\[
(A_i)_{\mu\nu}=
\begin{cases}
\gamma_{i-t}^{p_{i-t}}\alpha_{ij} &\mbox{if } \mu=\nu=(p_1,\ldots,p_t,j); \\
0 &\mbox{otherwise.}   \\
\end{cases}
\]

\item
For $2t+1\leq i\leq n$,
\[
(A_i)_{\mu\nu}=
\begin{cases}
\alpha_{ij} &\mbox{if } \mu=\nu=(p_1,\ldots,\ldots,p_t,j); \\
0 &\mbox{otherwise.}   \\
\end{cases}
\]
\end{enumerate}
\end{definition}

\begin{example}%\label{ex:Z}
If $m=6,n=5$ and  \[D=D_5(1/2,1/3)=
\begin{bmatrix}
0&0&-1/2&0&0\\
0&0&0&-1/3&0\\
1/2&0&0&0&0\\
0&1/3&0&0&0\\
0&0&0&0&0
\end{bmatrix},
\]
then $\gamma_1=-1,\gamma_2=e^{2\pi\sqrt{-1}/3},\sigma(D)=6$ and $l=1$. The subspace $Z_D\cong (\bS^1)^5$ consists of tuples $(A_1,\ldots,A_5)$ of the following forms
\setlength{\arraycolsep}{1pt}
\[
A_1=\begin{bmatrix}
    &\beta_{11}&&&&\\
    \;1\,&&&&&\\
    &&&\beta_{11}&&\\
    &&\,1\,&&&\\
    &&&&&\beta_{11}\\
    &&&&\,1\,&\\
    \end{bmatrix},
A_2=\begin{bmatrix}
    &&&&\beta_{21}&\\
    &&&&&\beta_{21}\\
    \;1\,&&&&&\\
    &\,1\,&&&&\\
    &&\,1\,&&&\\
    &&&\,1&&
    \end{bmatrix},
A_3=\begin{bmatrix}
    \alpha_{31}&&&&&\\
    &-\alpha_{31}&&&&\\
    &&\alpha_{31}&&&\\
    &&&-\alpha_{31}&&\\
    &&&&\alpha_{31}&\\
    &&&&&-\alpha_{31}\\
        \end{bmatrix},\]
\[A_4=\begin{bmatrix}
    \alpha_{41}&&&\\
    &\alpha_{41}&&\\
    &&\gamma_2\alpha_{41}&\\
    &&&\gamma_2\alpha_{41}\\
    &&&&\gamma_2^2\alpha_{41}&\\
    &&&&&\gamma_2^2\alpha_{41}
    \end{bmatrix}\text{ and }A_5=\begin{bmatrix}
    \alpha_{51}&&&&&\\
    &\alpha_{51}&&&&\\
    &&\alpha_{51}&&&\\
    &&&\alpha_{51}&&\\
    &&&&\alpha_{51}\\
    &&&&&\alpha_{51}
    \end{bmatrix}.
\]
In general for any positive integer $l$ and $m=l\sigma(D)$, each matrix $A_i$ in $(A_1,\ldots A_n)\in Z_D$ is a block sum of $l$ matrices of the form above.
\end{example}

By Theorem \ref{specAC}, any $(A_1,\ldots,A_n)\in B_n(U(m))_D$ is conjugate to an element in $Z_D$. In other words, the map
\begin{equation}\label{UZ}
U(m)\times Z_D\to B_n(U(m))_D
\end{equation}
given by the conjugation action $(M,(A_i))\mapsto(MA_iM^{-1})$ is surjective. The map is equivariant with respect to the $U(m)$-action given by left multiplication on the factor $U(m)$ of the domain and conjugation action on the target. This map (\ref{UZ}) is invariant under a few obvious actions on the domain corresponding to different choices of the vectors $v_j$ in Theorem \ref{specAC}.

\begin{enumerate}[(a)]

\item\label{S1action} For each $1\leq j \leq l$, an $\bS^1$-action on the domain is given by \[(M,(A_i))\mapsto(MI(j,\theta),(I(j,\theta)^{-1}A_iI(j,\theta)))=(MI(j,\theta),(A_i)),\]
where $\theta\in \bS^1$ and $I(j,\theta)\in U(m)$ is the matrix
\begin{equation}\label{Ijtheta}
I(j,\theta)_{\mu\nu}=
\begin{cases}
1 &\mbox{if } \mu=\nu=(p_1,\ldots,p_t,j'),j'\neq j\\
\theta &\mbox{if } \mu=\nu=(p_1,\ldots,p_t,j)\\
0 &\mbox{otherwise,}
\end{cases}
\end{equation}
obtained from the $m\times m$ identity matrix by multiplying all the $(p_1,\ldots,p_t,j)$-th columns, $0\leq p_i < |d_i|$, by $\theta\in \bS^1$. This action corresponds to replacing $v_j$ by $\theta v_j$ in Theorem \ref{specAC}.

\item\label{Zdiaction} For each $1\leq k\leq t,1\leq j \leq l$, an action on the domain is given by \[(M,(A_i))\mapsto(MA_k(j),(A_k(j)^{-1}A_iA_k(j))),\]
where $A_k(j)\in U(m)$ is defined by
\[
A_k(j)_{\mu\nu}=
\begin{cases}
1 &\mbox{if } \mu=(p_1,\ldots,p_k,\ldots,p_t,j),\nu=(p_1,\ldots,p_k-1,\ldots,p_t,j)\\
&\text{ or }\mu=\nu=(p_1,\ldots,p_k,\ldots,p_t,j'),j'\neq j;\\
\beta_{kj} &\mbox{if } \mu=(p_1,\ldots,p_{k-1},1,p_{k+1},\ldots,p_t,j),\\
& \nu=(p_1,\ldots,p_{k-1},|d_k|-1,p_{k+1},\ldots,p_t,j);\\
0 &\mbox{otherwise.}\\
\end{cases}
\]
This action corresponds to replacing $v_j$ by $A_kv_j$ in Theorem \ref{specAC}.
Note that $A_k(j)^{|d_k|}$ is equal to $I(j,\beta_{kj})$ in (\ref{Ijtheta}) above. Also, under the parametrization $Z_D\cong (\bT^n)^l$, this action multiplies $\alpha_{k+t,j}$ by $\gamma_k$ and keeps the other $\alpha_{i'j}$ and $\beta_{ij}$ fixed.

\item \label{permaction} A $\Sigma_l$-action on the domain is given by
\[(M,(A_i))\mapsto(MP_{\tau}^{-1},(P_{\tau}A_iP_{\tau}^{-1})),\]
where $\tau\in\Sigma_l$ and $P_{\tau}$ is the matrix
\[
P_{\tau}=
\begin{cases}
1 &\mbox{if } \mu=(p_1,\ldots,p_t,j),\nu=(p_1,\ldots,p_t,\tau(j))\\
0 &\mbox{otherwise.}
\end{cases}
\]
obtained by applying the permutation $\tau$ to the columns of the $m\times m$ identity matrix. This action corresponds to replacing $v_j$ by $v_{\tau(j)}$ in Theorem \ref{specAC}.
\end{enumerate}

Let $H$ be the group consisting of self homeomorphisms of $U(m)\times Z_D$ generated by the three types of actions above. Then the action induced by $\{I(j,\theta):1\leq j \leq l,\theta\in\bS^1\}$ in (\ref{S1action}) generates a normal subgroup of $H$ isomorphic to $\bT^l$ with quotient $H/\bT^l$ the wreath product
$\bZ_D\wr \Sigma_l\cong(\bZ_D)^l\rtimes \Sigma_l$ generated by the actions in (\ref{Zdiaction}) and (\ref{permaction}). Here $\bZ_D$ is the abelian group $\prod_{i=1}^t\bZ/|d_i|$. Hence, the map (\ref{UZ}) factors through
\[(U(m)\times Z_D)/H\cong \bigslant{\left(U(m)/\bT^l\right)\times (\bT^n)^l}{(\bZ_D\wr\Sigma_l)}
\cong \bigslant {\left[\bigslant{\left(U(m)/\bT^l\right)\times (\bT^n)^l}{(\bZ_D)^l}\right]}{\Sigma_l.}\]
%$$(U(m)\times Z_D)/H\cong \bigslant{\left(U(m)/\bT^l\right)\times (\bT^n)^l}{(\bZ_D\wr\Sigma_l)}\cong \bigslant {\left[\bigslant{\left(U(m)/\bT^l\right)\times (\bT^n)^l}{(\bZ_D)^l}\right]}{\Sigma_l.}$$
Our next theorem states that the induced factor map
\begin{equation}\label{UZqH}
\phi_D:\bigslant {\left[\bigslant{\left(U(m)/\bT^l\right)\times (\bT^n)^l}{(\bZ_D)^l}\right]}{\Sigma_l}\to B_n(U(m))_D
\end{equation}
of (\ref{UZ}) is a good approximation to $B_n(U(m))_D$.

\begin{theorem}\label{thm:BnD}
Let $2t\leq n$ and $D=D_n(d_1,\ldots,d_t)\in T(n,\bQ/\bZ)$. Suppose that $\sigma(D)=\prod{|d_i|}$ divides $m$ and $l=m/\sigma(D)$. Then the map $\phi_D$ in (\ref{UZqH}) is a rational homology equivalence for any $l\geq 1$ and a homeomorphism for $l=1$. The space $B_n(U(m))_D$ is path-connected and has rational cohomology
\begin{equation}\label{HBnmQ}
H^{\ast}(B_n(U(m))_D;\bQ)\cong H^{\ast}((U(m)/\bT^l)\times (\bT^n)^l;\bQ)^{\Sigma_l}.
\end{equation}
Also, $\phi_D$ induces a homeomorphism
$\bar{B}_n(U(m))_D:=B_n(U(m))_D/U(m)\cong (\bT^n)^l/\Sigma_l$
after passing to the $U(m)$-quotients.
\end{theorem}
%***(Should give a geometric explanation of $\phi$ and its domain)***
\begin{proof}
Since the map (\ref{UZ}) is surjective by Theorem \ref{specAC}, so is $\phi_D$. Let $(A_1,\ldots,A_n)\in B_n(U(m))_D$. Since $A_{t+1},A_{t+2},\ldots,A_n$ are pairwisely commuting, they can be simultaneously diagonalized. Let $\Lambda$
be the set of all $(n-t)$-tuples $\lambda=(\lambda_{t+1},\ldots,\lambda_n)\in \bT^{n-t}$ of eigenvalues of $(A_{t+1},A_{t+2},\ldots,A_n)$ and
$\Lambda_0=\{\lambda=(\lambda_{t+1},\ldots,\lambda_n)\in\Lambda|\, 0\leq \tfrac{1}{2\pi\sqrt{-1}}\log \lambda_{t+i}<\tfrac{1}{|d_i|}, \forall 1\leq i \leq t\}.$
Suppose $\lambda^1,\ldots,\lambda^s$ are all the distinct elements in $\Lambda_0$. Let $l_j=\dim E_{\lambda^j}$. Then
$\sum_{j=1}^sl_j=l$. One can show that the preimage
$\phi_D^{-1}(A_1,\ldots,A_n)\cong \prod_{j=1}^s U(l_j)/(T^{l_j}\rtimes\Sigma_{l_j})$,
which is $\bQ$-acyclic in general and a single point if $l=1$. Hence, $\phi_D$ is a rational homology equivalence for any $l\geq 1$ by the Vietoris-Begle mapping theorem and a homeomorphism for $l=1$. In particular, $B_n(U(m))_D$ is path-connected. Note that the self-homeomorphisms of $(U(m)/\bT^l)\times (\bT^n)^l$ given by the action of $\bZ_D$ in the domain of $\phi_D$ are homotopic to the identity map. Thus we have the isomorphism (\ref{HBnmQ}). Also, each preimage of $\phi_D$ is $U(m)$-transitive and so $\phi_D$ becomes a homeomorphism after passing to the $U(m)$-quotients. Hence,
\begin{align*}
\bar{B}_n(U(m))_D
&\cong \bigslant{\left[\bigslant{(\bT^n)^l}{(\bZ_D)^l}\right]}{\Sigma_l}
\cong \bigslant{\left[\bigslant{\bT^n}{\bZ_D}\right]^l}{\Sigma_l}\\
&\cong \bigslant{\left[\bigslant{(\bS^1)^n}{(\prod_{i=1}^t\bZ/|d_i|)}\right]^l}{\Sigma_l}\\
&\cong \bigslant{\left[\left(\prod_{i=1}^t\bigslant{\bS^1}{(\bZ/|d_i|)}\right)\times(\bS^1)^{n-t}\right]^l}{\Sigma_l}\\
&\cong \bigslant{\left[(\prod_{i=1}^t\bS^1)\times(\bS^1)^{n-t}\right]^l}{\Sigma_l}
=(\bT^n)^l/\Sigma_l.
\end{align*}
\end{proof}

%\subsection{Path connected components of $B_n(U(m))$.}
By Theorem \ref{specAC} and \ref{thm:BnD}, we know that for $D=D_n(d_1,\ldots,d_t)$, the space $B_n(U(m))_D$ is non-empty and path connected if $\sigma(D)$ divides $m$ and empty otherwise. We will extend the definition of $\sigma$ in such a way that the same statement holds for any $D\in\TnQZ$. This allows us to identify the path connected components of $B_n(U(m))$ and compute the number of them.

\begin{definition}
For any $n\times n$ matrix $A\in M_{n\times n}(\bQ/\bZ)$, the row space $R(A)$ of $A$ is the sub-module of $(\bQ/\bZ)^n$ generated by the rows of $A$ over $\bZ$. Define
\begin{equation}\label{def2sigma}
\sigma(A)=\sqrt{|R(A)|},
\end{equation}
where $|R(A)|$ is the cardinality of $R(A)$.
\end{definition}

%As we will show in the next lemma, this definition agrees with our previous definition of $\sigma(D)$ for $D=D_n(d_1,\ldots,d_t)$ in (\ref{def1sigma}).

\begin{lemma}\label{lem:sigmaD}
\begin{enumerate}[(a)]
\item\label{lem:sigmaDa} For $D=D_n(d_1,\ldots,d_t)\in \TnQZ$, $\sqrt{|R(D)|}=\prod_i|d_i|$.
\item\label{lem:sigmaDb} If $A,A'\in M_{n\times n}(\bQ/\bZ)$ with $A'=BAC$ for some $B,C\in GL(n,\bZ)$, then $\sigma(A)=\sigma(A')$.
\item\label{lem:sigmaDc} For $D\in\TnQZ$, $\sigma(D)$ is an integer. If $\sigma(D)$ divides $m$, then $D\in \Tnm$.
\end{enumerate}
\end{lemma}
\begin{proof}
Part (\ref{lem:sigmaDa}) is obvious. For part (\ref{lem:sigmaDb}), since $BA$ can be obtained from $A$ by elementary row operations, $R(BA)=R(A)$. Also, note that the multiplication map $v\mapsto vC$ for any row vectors $v\in(\bQ/\bZ)^n$ establishes an automorphism of $(\bQ/\bZ)^n$. Hence, $R(A')=R(BAC)\cong R(BA)=R(A)$ as $\bZ$-modules. This proves part (\ref{lem:sigmaDb}). For part (\ref{lem:sigmaDc}), note that $D\in T(n,\bZ/q)$ for some positive integer $q$. By \cite[Proposition 4.1]{RV}, $D$ is congruent to some matrix $D'=D_n(d_1,\ldots,d_t)\in T(n,\bZ/q)$. It follows from part (\ref{lem:sigmaDa}) and (\ref{lem:sigmaDb}) that $\sigma(D)=\sigma(D')=\prod_i|d_i|$ is an integer. Finally, if $D\in\TnQZ$ has an entry $d_{ij}\notin \bZ/m$, then $i$-th and $j$-th rows of $D$ generate a $\bZ$-submodule of $R(D)$ with cardinality a multiple of $|d_{ij}|^2$. This implies $|d_{ij}|$ divides $\sigma(D)$ and so $\sigma(D)$ does not divide $m$.
\end{proof}

\begin{corollary}\label{cor:BnUgenD}
For $D\in\TnQZ$, the space $B_n(U(m))_D$ is non-empty and path connected if $\sigma(D)$ divides $m$, and is empty otherwise. Hence, $B_n(U(m))$ can be expressed as the disjoint union
$B_n(U(m))=\coprod\limits_{\substack{D\in T(n,\bZ/m)\\ \sigma(D)|m}}B_n((U(m))_D$ of path connected components.
\end{corollary}
\begin{proof}
By \cite[Proposition 4.1]{RV}, $D$ is congruent to some matrix $D'=D_n(d_1,\ldots,d_t)\in T(n,\bQ/\bZ)$. Since $B_n(U(m))_D\cong B_n(U(m))_{D'}$ by Proposition \ref{congDhomeo} and $\sigma(D)=\sigma(D')$ by Lemma \ref{lem:sigmaD}, the first statement of the corollary follows from the corresponding statement for $D'$ in Theorem \ref{specAC} and \ref{thm:BnD}. Since the map $\rho$ in (\ref{def:phi}) is continuous, $B_n(U(m))$ can be expressed as the disjoint union of path connected components above.
\end{proof}

By Corollary \ref{cor:BnUgenD}, the number of path-connected components of $B_n(U(m))$ is equal to
\begin{equation}\label{Nnm}
N(n,m)=|\{D\in T(n,\bZ/m) : \sigma(D)\text{ divides }m\}|.
\end{equation}

We will derive formulas for $N(n,m)$. To do this, we need to introduce some notation. For positive integers $k,n \geq 1$, let $J'(n,k)$ be the set of all partitions of $k$ with at most $n/2$ parts and $J(n,k)=\cup_{1\leq k'\leq k} J'(n,k')$. Every element in $J(n,k)$ is a partition which can be written uniquely in the form
\[
\underbrace{a_1+\ldots +a_1}_{t_1\; {a_1}'s}+\underbrace{a_2+\ldots +a_2}_{t_2\; {a_2}'s}+\ldots+ \underbrace{a_j+\ldots +a_j}_{t_j\; {a_j}'s}
\]
with $a_1>a_2>\ldots>a_j>0,\sum t_ia_i \leq k$ and $\sum t_i \leq n/2$. For instance, $J(5,4)$ consists of the eight partitions
$1,2,1+1,3,2+1,4,3+1$ and $2+2$.

For the rest of this section, we will denote the elements of $\bZ/m$ by $0,1,2,\ldots, m-1$ rather than $[0],[\tfrac{1}{m}],[\tfrac{2}{m}],\ldots,[\tfrac{m-1}{m}]$. Suppose that $q=p^k$ is a power of a prime $p$ with $k\geq 1$ and $\alpha\in J(n,k)$ is the partition $\sum_{i=1}^j t_ia_i$. Let $b_i=k-a_i$. Define $D_{\alpha}\in T(n,\bZ/q)$ to be the matrix
\[
D_{\alpha}=D_n(\underbrace{p^{b_1},\ldots, p^{b_1}}_{t_1\; {p^{b_1}}'s},\underbrace{p^{b_2},\ldots, p^{b_2}}_{t_2\; {p^{b_2}}'s},\ldots, \underbrace{p^{b_j},\ldots, p^{b_j}}_{t_j\; {p^{b_j}}'s})
\]
Note that $\sigma(D_{\alpha})=\prod (p^{a_i})^{t_i}=p\,^{\sum t_ia_i}$ divides $q$. %Consider the action of $GL(n,\bZ/q)$ on $T(n,\bZ/q)$ given by $A\cdot D=A^TDA$. Two matrices in $T(n,\bZ/q)$ are said to be similar if they are in the same orbit of this action.
\begin{proposition}
Fix $k,n\geq 1$. Let $p$ be a prime, $q=p^k$ and $\alpha\in J(n,k)$. Then the number of matrices in $T(n,\bZ/q)$ congruent to $D_{\alpha}$ is given by
\[N_p(\alpha)=\frac{p\,^{\sum_{i=1}^j t_ia_i(s_{i-1}+s_i-1)}\cdot\prod_{l=s_j+1}^n\left(1-\frac{1}{p^l}\right)}{\prod_{i=1}^j\prod_{l=1}^{t_i}\left(1-\frac{1}{p^{2l}}\right)}\]
where $s_i=n-2\sum_{l=1}^it_l$.
The number of matrices $D\in T(n,\bZ/q)$ with $\sigma(D)$ divides $q$ is equal to
\begin{equation}\label{Nnq}
N(n,q)=1+\sum_{\alpha\in J(n,k)} N_p(\alpha).
\end{equation}
\end{proposition}
\begin{remark} The summand 1 in (\ref{Nnq}) accounts for the zero matrix $\mathbf{0}\in T(n,\bZ/q)$ in the definition of $N(n,q)$ in (\ref{Nnm}). This corresponds to the path connected component $B_n(U(m))_{\mathbf{0}}$ of $B_n(U(m))$ consisting of commuting matrices in Corollary \ref{cor:BnUgenD}.
\end{remark}
\begin{proof}
Let $G=GL(n,\bZ/q)$ and $G_{D_\alpha}$ be the stabilizer of $D_{\alpha}$. Then the number of matrices in $T(n,\bZ/q)$ congruent to $D_{\alpha}$ is equal to $|G|/|G_{D_\alpha}|$. Recall the well-known result that
\begin{equation}\label{GLnZqsize}
 |G|=q^{n^2}\prod_{l=1}^n\left(1-\tfrac{1}{p^l}\right).
\end{equation}
To count $|G_{D_\alpha}|$, consider the skew-symmetric bilinear form $\omega:(\bZ/q)^n\times (\bZ/q)^n\to\bZ/q$ associated to $D_{\alpha}=(d_{ij})$. Let $v_1,\ldots v_n$ be the columns of a matrix $A\in GL(n,\bZ/q)$. Then $A^TD_{\alpha}A=D_{\alpha}$ if and only if $\omega(v_i,v_j)=d_{ij}$ for all $1\leq i,j\leq n$. Hence, $|G_{D_{\alpha}}|$ is equal to the number of ordered bases $v_1,v_2,\ldots v_n$ of $(\bZ/q)^n$ satisfying $\omega(v_i,v_j)=d_{ij}$. We will count them by picking the basis vectors in the order $v_1,v_{t+1},v_2,v_{t+2},\ldots,v_t,v_{2t},v_{2t+1},v_{2t+2},\ldots,v_n$. To pick $v_1$, there are $\left(q^{2t_1}-(\tfrac{q}{p})^{2t_1}\right)q^{s_1}=q^n\left(1-\tfrac{1}{p^{2t_1}}\right)$ choices. Without loss of generality, we may assume $v_1=e_1$. Then it is obvious that there are $p^{b_1}q^{n-1}$ choices for $v_{t+1}$. Again we may assume $v_{t+1}=e_{t+1}$. Similar arguments show that there are $(p^{b_1})^2\left(q^{2t_1-2}-(\tfrac{q}{p})^{2t_1-2}\right)q^{s_1}=(p^{b_1})^2q^{n-2}\left(1-\tfrac{1}{p^{2t_1-2}}\right)$ choices for $v_2$, then $(p^{b_1})^3q^{n-3}$ choices for $v_{t+2}$ and so on. Hence, in total there are
%\begin{align*}
%&(p^{a_1})^{1+2+\ldots+(2t_1-1)}q^{n+(n-1)+\ldots (n-2t+1)}\prod_{i=1}^{t_1}\left(1-\tfrac{1}{p^{2i}}\right)\\
%=\,&p^{a_1t_1(2t_1-1)}q^{t_1(2n-2t_1+1)}\prod_{i=1}^{t_1}\left(1-\tfrac{1}{p^{2i}}\right)
%\end{align*}
\[
(p^{b_1})^{1+2+\ldots+(2t_1-1)}q^{n+(n-1)+\ldots (n-2t_1+1)}\prod_{l=1}^{t_1}\left(1-\tfrac{1}{p^{2l}}\right)
=(p^{b_1})^{t_1(2t_1-1)}q^{t_1(s_0+s_1+1)}\prod_{l=1}^{t_1}\left(1-\tfrac{1}{p^{2l}}\right)
\]
choices for the vectors $v_1,v_{t+1},\ldots, v_{t_1},v_{t+t_1}$.
Similarly, for the next $2t_2$ vectors $v_{t_1+1}$, $v_{t+t_1+1}$, $\ldots$, $v_{t_1+t_2}$, $v_{t+t_1+t_2}$,
there are
$(p^{b_1})^{4t_1t_2}(p^{b_2})^{t_2(2t_2-1)}q^{t_2(s_1+s_2+1)}\prod_{l=1}^{t_2}\left(1-\tfrac{1}{p^{2l}}\right)$
choices. Proceeding in this way, after the vectors $v_a,v_{t+a},1 \leq a\leq t_1+t_2+\ldots +t_{i-1}$ are chosen, there are
\begin{equation}\label{vichoices}
(p^{b_1})^{4t_1t_i}\ldots (p^{b_{i-1}})^{4t_{i-1}t_i} (p^{b_i})^{t_i(2t_i-1)}q^{t_i(s_{i-1}+s_i+1)}\prod_{l=1}^{t_i}\left(1-\tfrac{1}{p^{2l}}\right)
\end{equation}
choices for the next $2t_i$ vectors $v_a,v_{t+a}$, where $t_1+t_2+\ldots +t_{i-1}+1\leq a\leq t_1+t_2+\ldots +t_i$. Finally, after the first $2(t_1+\ldots+t_j)$ vectors are chosen, in total there are
\begin{equation}\label{remainchoices}
(p^{b_1})^{2t_1s_j}\ldots (p^{b_j})^{2t_js_j}q^{{s_j}^2}\prod_{l=1}^{s_j}\left(1-\tfrac{1}{p^l}\right)
\end{equation}
choices for the remaining $s_j$ vectors $v_{t_1+\ldots +t_j+1},\ldots,v_n$.
Hence, taking the product of (\ref{vichoices}) for $i=1,\ldots,j$ and (\ref{remainchoices}), we get
\begin{equation}\label{GDsize}
|G_{D_{\alpha}}|=q^e\left(\prod_{i=1}^j(p^{b_i})^{f_i}\right)
\left(\prod_{i=1}^j\prod_{l=1}^{t_i}\left(1-\tfrac{1}{p^{2l}}\right)\right)\left(\prod_{l=1}^{s_j}\left(1-\tfrac{1}{p^l}\right)\right)
\end{equation}
where $e=s_j^2+\sum_{i=1}^jt_i(s_{i-1}+s_i+1)$ and
\begin{align*}
 f_i&=t_i\left(2s_j+2t_i-1+4\textstyle\sum_{l=i+1}^jt_{l}\right)\\
&=t_i\left(2n-4\textstyle\sum_{l=1}^jt_l+2t_i-1+4\textstyle\sum_{l=i+1}^jt_{l}\right)\\
&=t_i\left(2n-1-2\textstyle\sum_{l=1}^{i-1}t_l-2\textstyle\sum_{l=1}^{i}t_l\right)\\
&=t_i(s_{i-1}+s_{i}-1).
\end{align*}
Note that
\begin{align*}
e+\textstyle\sum_{i=1}^j f_i
&=s_j^2+\textstyle\sum_{i=1}^jt_i(s_{i-1}+s_i+1)+\sum_{i=1}^j t_i(s_{i-1}+s_{i}-1)\\
&=s_j^2+\textstyle\sum_{i=1}^j2t_i(s_{i-1}+s_i)\\
&=s_j^2+\textstyle\sum_{i=1}^j(s_{i-1}-s_i)(s_{i-1}+s_i)\\
&=s_j^2+\textstyle\sum_{i=1}^j({s_{i-1}}^2-{s_i}^2)\\
&=s_0^2\\
&=n^2.
\end{align*}
Hence, dividing (\ref{GLnZqsize}) by (\ref{GDsize}), we get
\begin{align*}
N_p(\alpha)=\frac{|G|}{|G_{D_{\alpha}}|}
&=\frac{\left(\prod_{i=1}^jq^{ f_i}\right)\left(\prod_{l=s_j+1}^n\left(1-\frac{1}{p^l}\right)\right)}{\left(\prod_{i=1}^j(p^{b_i})^{f_i}\right)\left(\prod_{i=1}^j\prod_{l=1}^{t_i}\left(1-\frac{1}{p^{2l}}\right)\right)}\\
&=\frac{\left(\prod_{i=1}^j(p^{a_i})^{ f_i}\right)\left(\prod_{l=s_j+1}^n\left(1-\frac{1}{p^l}\right)\right)}{\prod_{i=1}^j\prod_{l=1}^{t_i}\left(1-\frac{1}{p^{2l}}\right)}.
\end{align*}
%***(There is a geometric way to realize the above computation)***

The last statement of the proposition follows easily from the fact that any non-zero $D\in T(n,\bZ/q)$ with $\sigma(D)$ divides $q$ is congruent to $D_{\alpha}$ for a unique $\alpha\in J(n,k)$.
\end{proof}

%***(Can things be realized as a generating function?)***

\begin{example}
Let $\alpha,\beta,\gamma$ be the partition $1,2,1+1$ respectively. For the partitions $\alpha,\beta$, $j=1,t_1=1,s_0=n$ and $s_1=n-2$.
\[
N_p(\alpha)=\frac{p^{2n-3}\left(1-\frac{1}{p^{n-1}}\right)\left(1-\frac{1}{p^n}\right)}{1-\frac{1}{p^2}}=\frac{(p^{n-1}-1)(p^n-1)}{p^2-1}
\]
\[
N_p(\beta)=\frac{p^{2(2n-3)}\left(1-\frac{1}{p^{n-1}}\right)\left(1-\frac{1}{p^n}\right)}{1-\frac{1}{p^2}}=\frac{p^{2n-3}(p^{n-1}-1)(p^n-1)}{p^2-1}
\]
For the partition $\gamma$, $j=1,t_1=2,a_1=1,s_0=n$ and $s_1=n-4$.
\[
N_p(\gamma)=\frac{p^2(p^{n-3}-1)(p^{n-2}-1)(p^{n-1}-1)(p^n-1)}{(p^2-1)(p^4-1)}
\]
Hence, for a prime number $p$,
\[N(n,p)=1+N_p(\alpha)=1+\frac{(p^{n-1}-1)(p^n-1)}{p^2-1}.\]
This agrees with \cite[Theorem 1]{ACG}. For $n\geq 4$,
\begin{align*}
N(n,p^2)&=1+N_p(\alpha)+N_p(\beta)+N_p(\gamma)\\
&=1+\frac{(p^{n-1}-1)(p^n-1)(p^{2n+1}-p^n-p^{n-1}+p^4+p^2-1)}{(p^2-1)(p^4-1)}.
\end{align*}
\end{example}

We will now look at $N(n,m)$ in the general case where $m$ is not necessarily a prime power. Let $m_1,m_2$ be relatively prime positive integers and $m=m_1m_2$. The inclusions
$\bZ/m_i\cong\bZ[\frac{1}{m_i}]/\bZ\subset\bZ[\frac{1}{m}]/\bZ\cong \bZ/m$ establish direct sum decompositions $\bZ/m_1\oplus \bZ/m_2 \cong \bZ/m$ and
\begin{equation}\label{Tdecomp}
T(n,\bZ/m_1)\oplus T(n,\bZ/m_2)\stackrel{\cong}{\longrightarrow}T(n,\bZ/m).
\end{equation}

\begin{proposition}
Let $m=m_1m_2$ with $m_1,m_2$ be relatively prime positive integers. Suppose $D_i\in T(n,\bZ/m _i)$ for $i=1,2$. Then $D_1+D_2\in T(n,\bZ/m)$ satisfies $\sigma(D_1+D_2)=\sigma(D_1)\sigma(D_2)$. Moreover, $N(n,m)=N(n,m_1)N(n,m_2)$.
\end{proposition}
\begin{proof}
We will show that $R(D_1)\oplus R(D_2)\cong R(D_1+D_2)$ as $\bZ$-submodules of $(\bZ/m_1)^n\oplus (\bZ/m_2)^n\cong (\bZ/m)^n$. For $1\leq j \leq n$, let $u_j\in(\bZ/m_1)^n,v_j\in(\bZ/m_2)^n$ be the $j$-th row of $D_1, D_2$ respectively. Then each row $u_i+v_i$ of $D_1+D_2$ is in $R(D_1)\oplus R(D_2)$ and hence $R(D_1+D_2) \subset R(D_1)\oplus R(D_2)$. For the other inclusion, pick integers $a,b$ so that $am_1\equiv 1\pmod{m_2}$ and $bm_2\equiv 1\pmod{m_1}$. Then $am_1(u_i+v_i)=v_i$ since $m_1u_i=m_2v_i=\vec{0}\in(\bZ/m)^n$. Similarly $bm_2(u_i+v_i)=u_i$. It follows that $R(D_1)\oplus R(D_2)\subset R(D_1+D_2)$. Hence $R(D_1)\oplus R(D_2)\cong R(D_1+D_2)$ and $\sigma(D_1+D_2)=\sigma(D_1)\sigma(D_2)$. To prove the last part of the proposition, note that since $R(D_i)\subset(\bZ/m_i)^n$, any prime divisor of $\sigma(D_i)$ divides $m_i$. Therefore, $\sigma(D_1+D_2)=\sigma(D_1)\sigma(D_2)$ divides $m=m_1m_2$ if and only if $\sigma(D_i)$ divides $m_i$ for each $i=1,2$. We conclude from (\ref{Tdecomp}) that $N(n,m)=N(n,m_1)N(n,m_2)$.
\end{proof}

\begin{corollary}
Let $p_1,\ldots,p_r$ be $r$ distinct prime numbers and $m=p_1^{k_1}p_2^{k_2}\ldots p_r^{k_r}$. Then $B_n(U(m))$ has
\[N(n,m)=\prod_{i=1}^r\left(1+\sum_{\alpha\in J(n,k_i)} N_{p_i}(\alpha)\right)\]
path-connected components.
\end{corollary}

We would like to close this section by giving the map $\phi_D$ in (\ref{UZqH}) a geometric interpretation analogous to that of (\ref{phiabelian}) and a generalization to any $D\in\TnQZ$ such that $B_n(U(m))_D$ is non-empty. Define $B_D=B_n(U(\sigma(D)))$ for any $D\in\TnQZ$. Suppose that $D=D_n(d_1,\ldots,d_t)$. By Theorem \ref{thm:BnD}, $B_D\cong((U(\sigma(D))/\bT^1)\times \bT^n)/\bZ_D$. The domain of $\phi_D$ can thus be expressed in terms of $B_D$ as
\begin{align}
&\bigslant {\left[\bigslant{\left(U(m)/\bT^l\right)\times (\bT^n)^l}{(\bZ_D)^l}\right]}{\Sigma_l}\notag\\
\cong&\bigslant {\left[\bigslant{U(m)\times_{U(\sigma(D))^l}(U(\sigma(D))/\bT^1)^l\times (\bT^n)^l}{(\bZ_D)^l}\right]}{\Sigma_l}\notag\\
\cong&\bigslant{\left[U(m)\times_{U(\sigma(D))^l}\left(\bigslant{(U(\sigma(D))/\bT^1)\times \bT^n}{\bZ_D}\right)^l\right]}{\Sigma_l}\notag\\
\cong&\,\bigslant{U(m)\times_{U(\sigma(D))^l}B_D\,^l}{\Sigma_l}\label{domainphi}.
\end{align}
In this form, each point of the domain of $\phi_D$ is an unordered $l$-tuple of pairwisely orthogonal $\sigma(D)$-dimensional complex subspaces of $\bC^m$ with a $D$-commuting tuple of unitary automorphism defined on each of these subspaces. The map $\phi_D$ simply glues the automorphisms using direct sum to produce a $D$-commuting tuple of unitary matrices. With this interpretation one can readily generalize $\phi_D$ to any $D\in\TnQZ$ with $m=l\sigma(D)$ and obtain the following corollary of Theorem \ref{thm:BnD}.

\begin{corollary}\label{cor:BnDq}
Suppose that $D\in\TnQZ$ and $m=l\sigma(D)$. The map
\begin{equation}\label{phigeneralD}
\phi_D:\bigslant{U(m)\times_{U(\sigma(D))^l}B_D\,^l}{\Sigma_l}\to B_n(U(m))_D.
\end{equation}
induces an isomorphism in rational cohomology.
Moreover, $\phi_D$ induces a homeomorphism
$\bar{B}_n(U(m))_D\cong \bar{B}_n(\sigma(D))_D\,^l/\Sigma_l$
after passing to quotients by the action of $U(m)$.
\end{corollary}
\begin{proof}
Suppose $D\in\TnQZ$. By Proposition \ref{congDhomeo} and \cite[Proposition 4.1]{RV}, we can pick an automorphism of $F_n$ which gives canonical homeomorphisms $B_n(U(m)_D\cong B_n(U(m)_{D'}$ and $B_D\cong B_{D'}$ for some $D'=D_n(d_1,\ldots,d_t)$. These homeomorphisms induce the vertical homeomorphisms in the commutative diagram
\[
\xymatrix{
\bigslant{U(m)\times_{U(\sigma(D))^l}B_D\,^l}{\Sigma_l}\ar[r]^(.64){\phi_D}\ar[d]^{\cong}& B_n(U(m))_D\ar[d]^{\cong}\,\\
\bigslant{U(m)\times_{U(\sigma(D'))^l}B_{D'}\,^l}{\Sigma_l} \ar[r]^(.64){\phi_{D'}} &B_n(U(m))_{D'}.
}
\]
Hence, $\phi_D$ satisfies the properties stated in the corollary because $\phi_{D'}$ does by (\ref{domainphi}) and Theorem \ref{thm:BnD}. This proves the corollary.
\end{proof}

\section{The rank one case}

Let $\Gamma$ be the central extension
\begin{equation}\label{gammacentralext}
1\to\bZ^r\to\Gamma\to \bZ^n\to 1.
\end{equation}
with $k$-invariant $\omega=(\omega_1,\dots, \omega_r)\in H^2(\mathbb{Z}^n; \mathbb{Z}^r)
\cong \left( H^2(\mathbb{Z}^n;\mathbb{Z})\right)^r,$
where
\begin{equation}\label{gammalcentralext}
\omega_l = \sum_{1\leq i < j\leq n } \omega_{i,j}^le^*_i\wedge
e^*_j \in H^2(\mathbb{Z}^n; \mathbb{Z}).
\end{equation}
In this description,  $\{e_i\}_{i=1}^n$ are the standard generators of $\mathbb{Z}^n$,
$\{e_i^*\}_{i=1}^n$ are the dual generators of the cohomology ring
$H^*(\mathbb{Z}^n;\mathbb{Z})$ and $\omega_{i,j}^l\in \mathbb{Z}$.
The group $\Gamma$ corresponding to this $k$-invariant
$\omega$ is given in terms of generators and relations by
\[\Gamma=\langle a_1, \dots, a_n, x_1, \dots, x_r :
[a_i,a_j]=\prod_{l=1}^rx_l^{\omega_{i,j}^l}, x_i \text{ is central}\rangle.\]
A group homomorphism from $\Gamma$ to $U(m)$ is determined by the images of the generators of $\Gamma$. By abuse of notation, we sometimes write $(X_1,\ldots,X_r,A_1,\ldots,A_n)\in Hom(\Gamma,U(m))$ to represent the homomorphism $f:\Gamma\to U(m)$ with $f(a_i)=A_i$ and $f(x_j)=X_j$ for $1\leq i\leq n, 1\leq j\leq r$.

For the rest of the paper, we will study the space $\Hom(\Gamma,U(m))$. In this section, our focus will be on the special case of rank $r=1$. The first step to analyze this space is to simplify the expression of the $k$-invariant $\omega\in H^2(\mathbb{Z}^n; \mathbb{Z})$ of the central extension. The following proposition allows us to work with central extensions with a particularly simple form of $k$-invariants. The proof is exactly analogous to that of \cite[Proposition 4.1]{RV} for the case of $(\bZ/m)$-valued skew-symmetric forms on $\bZ^n$ and we will leave it to the readers.

\begin{proposition}\label{skew}
Let $\omega:\bZ^n\times\bZ^n\to\bZ$ be a skew-symmetric bilinear form on $\bZ^n$. Then there exist $t\leq n/2$ and a $\bZ$-module basis $e_1,e_2,\ldots,e_n$ of $\bZ^n$ such that

\[\omega=c_1e_1^{\ast}\wedge e_{t+1}^{\ast}+\ldots+c_t e_t^{\ast}\wedge e_{2t}^{\ast},\]
where $c_1,\ldots, c_t$ are positive integers with $c_i|c_{i+1}$ for $i=1,\ldots, t-1$.
\end{proposition}

Thanks to Proposition \ref{skew}, we can assume throughout this section that $\Gamma$ is the central extension
\begin{equation}\label{Gammak}
 1\to \bZ\to \Gamma \to \bZ^n\to 1
\end{equation}
with $k$-invariant $\omega=\sum_{i=1}^t c_ie_i^{\ast}\wedge e_{t+i}^{\ast}$, where $0\leq 2t\leq n$ and $c_i>0 $ for $1\leq i \leq t$. In this case, an $(n+1)$--tuple of $m\times m$ unitary matrices $(X,A_1,\ldots,A_n)\in\Hom(\Gamma,U(m))$ if and only if for $1\leq i\leq i'\leq n$,
\begin{align}
&[X,A_i]=I_m;\label{XAcomm} \\
&[A_i,A_{i'}]=
\begin{cases}
X^{c_i} &\mbox{if } i'=i+t,1\leq i \leq t; \\
I_m &\mbox{otherwise.}   \\
\end{cases}\label{AAcomm}
\end{align}

\begin{lemma}\label{Xeigen}
Let $1\to\bZ\to \Gamma \to \bZ^n\to 1$ be a central extension with non-zero $k$-invariant $\omega$. Then there exists a positive integer $L$ such that for any $(X,A_1,\ldots,A_n)\in \text{Hom}(\Gamma,U(m))$, each eigenvalue of $X$ is a $L$-th root of unity.
\end{lemma}
\begin{proof}
Since $\omega\neq 0$, $[A_i,A_j]=X^{\omega_{ij}}$ with $\omega_{ij}> 0$ for some $1\leq i,j\leq n$. Suppose that $\lambda$ is an eigenvalue of $X$ and the associated eigenspace $E_{\lambda}$ is $k$ dimensional. Since $A_i,A_j$ commutes with $X$, they restrict to unitary automorphisms on $E_{\lambda}$. $[A_i|_{E_{\lambda}},A_j|_{E_{\lambda}}]=X^{\omega_{ij}}|_{E_{\lambda}}=\lambda^{\omega_{ij}}Id_{E_{\lambda}}$. By taking determinant, we have $\lambda^{k\omega_{ij}}=1$. Therefore, we can take $L=m!\,\omega_{ij}$.
\end{proof}

Consider the map $\chi:U(m)\to \bC[z]$ which sends a unitary matrix to its characteristic polynomial and the restriction map $\res:\Hom(\Gamma,U(m))\to \Hom(\bZ,U(m))\cong U(m)$ which sends $(X,A_1,
\ldots,A_n)$ to $X$. For a complex polynomial $p(z)$, define
$U(m)_{p(z)}=\chi^{-1}(p(z))$, and
$\Hom(\Gamma,U(m))_{p(z)}=(\chi\circ res)^{-1}(p(z))$.
The map $\res$ restricts to a map
\[
\res_{p(z)}:\Hom(\Gamma,U(m))_{p(z)}\to U(m)_{p(z)}.
\]
Note that if $\Gamma$ is non-abelian, the range of the map $\chi\circ \res$ is a discrete set by Lemma \ref{Xeigen}. Hence, by continuity, each $\Hom(\Gamma,U(m))_{p(z)}$ is both an open and closed subset of $\Hom(\Gamma,U(m))$.

Let $p(z)=\prod_{j=1}^s(z-\lambda_j)^{m_j}\in\bC[z]$ with distinct roots $\lambda_1,\ldots,\lambda_s\in\bS^1$. Let $X\in U(m)_{p(z)}$ and $(X,A_1,\ldots,A_n)\in (\res_{p(z)})^{-1}(X)$. There is an orthogonal decomposition
\[
\bC^m=E_{\lambda_1}\oplus E_{\lambda_2} \oplus \ldots \oplus E_{\lambda_s}
\]
of $\bC^m$ into eigenspaces of $X$. $X$ restricts to complex multiplication by $\lambda_j$ on $E_{\lambda_j}$ and $X|_{E_{\lambda_j}}$ has characteristic polynomial $(z-\lambda_j)^{m_j}$. Since $A_i$ commutes with $X$, $A_i$ also preserves each of these eigenspaces ${E_{\lambda_j}}$. All these $n+1$ matrices $X, A_1,\ldots, A_n$ restrict to unitary automorphisms on the eigenspace $E_{\lambda_j}$ and their restrictions on $E_{\lambda_j}$ satisfy commutation relations similar to those in (\ref{XAcomm}) and (\ref{AAcomm}). Hence, for $1\leq i\leq i' \leq n$ and $1\leq j \leq s$,
\begin{align}
&X|_{E_{\lambda_j}}= \lambda_jId_{E_{\lambda_j}} \text{ is central;}\label{XjAcomm}\\
&[A_i|_{E_{\lambda_j}},A_{i'}|_{E_{\lambda_j}}]=
\begin{cases}
\lambda_j^{c_i}Id_{E_{\lambda_j}} &\mbox{if } i'=i+t,1\leq i \leq t; \\
Id_{E_{\lambda_j}} &\mbox{otherwise.}
\end{cases} \label{AAjcomm}
\end{align}
Under an unitary isomorphism $E_{\lambda_j}\cong \bC^{m_j}$, these relations between the restrictions of $X,A_1,\ldots A_n$ on $E_{\lambda_j}$ are the same as those among an $(n+1)$-tuple in $\Hom(\Gamma,U(m_j))_{(z-\lambda_j)^{m_j}}$. Since $A_1,\ldots,A_n$ are uniquely determined by their restrictions on the eigenspaces $E_{\lambda_j}$, this implies
\begin{equation}\label{respzinv}
(\res_{p(z)})^{-1}(X)\cong\prod_{j=1}^s{\Hom(\Gamma,U(m_j))_{(z-\lambda_j)^{m_j}}}.
\end{equation}

For instance, if we take $X$ to be the matrix
\begin{equation}\label{X0}
X_0=
  \begin{bmatrix}
    \lambda_1I_{m_1} & 0 & \ldots &  0\\
    0 & \lambda_2I_{m_2} & \ldots  & 0\\
    \vdots & \vdots & \ddots &  \vdots\\
    0 & 0 & \ldots &  \lambda_jI_{m_j}\\
  \end{bmatrix}\in U(m)_{p(z)},
\end{equation}
then $(\res_{p(z)})^{-1}(X_0)$ consists of all $(n+1)$-tuples $(X_0,A_1,\ldots,A_n)$ such that each $A_i$ has the form
\[
A_i=
  \begin{bmatrix}
    A_{i1} & 0 & \ldots &  0\\
    0 & A_{i2} & \ldots  & 0\\
    \vdots & \vdots & \ddots &  \vdots\\
    0 & 0 & \ldots &  A_{is}\\
  \end{bmatrix}\in U(m)
\]
where $(\lambda_jI_{m_j},A_{1j},A_{2j},\ldots,A_{nj})\in \Hom(\Gamma,U(m_j))_{(z-\lambda_j)^{m_j}}$ for each $j$.

Our next theorem says that $\res_{p(z)}:\Hom(\Gamma,U(m))_{p(z)}\to U(m)_{p(z)}$ is a fiber bundle with fiber $\res^{-1}(X_0)$ homeomorphic to product of spaces of almost commuting tuples of unitary matrices.

\begin{theorem}\label{thmfib}
Let $\Gamma$ be a central extension as in (\ref{Gammak}) and $p(z)=\prod_{j=1}^s(z-\lambda_j)^{m_j}$, where $\lambda_1,\ldots,\lambda_s\in\bS^1$ are its distinct roots. Then there are $U(m)$-equivariant homeomorphisms
\begin{align*}
\Hom(\Gamma,U(m))_{p(z)}
&\cong U(m)\times_{\prod_{j=1}^sU(m_j)}\left(\prod_{j=1}^s\Hom(\Gamma,U(m_j))_{(z-\lambda_j)^{m_j}}\right)\\
&\cong U(m)\times_{\prod_{j=1}^sU(m_j)}\left(\prod_{j=1}^sB_n(U(m_j))_{D_n(-c_1q_j,\ldots,-c_tq_j)}\right)
\end{align*}
where $q_j=\frac{1}{2\pi\sqrt{-1}}\log{\lambda_j}$. Also, $U(m)_{p(z)}\cong U(m)/(\textstyle\prod_{j=1}^sU(m_j))$ is a flag manifold and $\res_{p(z)}:\Hom(\Gamma,U(m))_{p(z)}\to U(m)_{p(z)}$ is a fiber bundle induced by the collapse map of $\prod_{j=1}^s\Hom(\Gamma,U(m_j))_{(z-\lambda_j)^{m_j}}$.
\end{theorem}

The proof of Theorem \ref{thmfib} makes use of the following lemma in \cite[Proposition 2.3.2]{Bredon}.

\begin{lemma}\label{lem:fiberbundle}
Let $G$ be a compact Lie group and $f:Y\to Z$ be a $G$-map between compact $G$-spaces. Suppose that the $G$-action on $Z$ is transitive with stabilizer subgroup $H\subset G$ at a point $z_0\in Z$. Then $Y_0=f^{-1}(z_0)$ is an $H$-space, $Y\cong G\times_HY_0$ as $G$-spaces and $f$ is the fiber bundle $G\times_HY_0\to G/H\cong Z$ induced by the collapse map of $Y_0$.
\end{lemma}

\begin{proof}[of Theorem~{\rm\ref{thmfib}}]
It is clear that $\res_{p(z)}:\Hom(\Gamma,U(m))_{p(z)}\to U(m)_{p(z)}$ is equivariant with respect to the conjugation action of $U(m)$. Also, note that this conjugation action on $U(m)_{p(z)}$ is transitive with the stabilizer subgroup of $X_0$ in (\ref{X0}) equal to $\prod_{j=1}^sU(m_j)$. By Lemma \ref{lem:fiberbundle} and (\ref{respzinv}), we obtain the first homeomorphism and prove that $\res_{p(z)}$ is induced by the collapse map of $\prod_{j=1}^s\Hom(\Gamma,U(m_j))_{(z-\lambda_j)^{m_j}}$. The second homeomorphism follows immediately from (\ref{XjAcomm}), (\ref{AAjcomm}) and the definition of $B_n(U(m_j))_{D_n(-c_1q_j,\ldots,-c_tq_j)}$.
\end{proof}

%\begin{theorem}[Just a theorem in section 3 but in another form. To be deleted or moved to introduction.]
%Let $\Gamma$ be as in (\ref{Gammak}). If $m=l\times \prod{k^{\lambda^{c_i}}}$, then there is a rational equivalence
%\begin{equation*}
%\bigslant{\left[\bigslant{\left(U(m)/\bT^l\right)\times (\bT^n)^l}{(\prod_{i=1}^t\bZ/k_{\lambda^{c_i}})^l}\right]}
%{\Sigma_l}\to \Hom(\Gamma,U(m))_{(z-\lambda_i)^m}.
%\end{equation*}
%If $K=m$, then the map is a homeomorphism.
%\end{theorem}

Finally, using results from the last section, we can write down the image of the map $\res$ and the connected components of $\Hom(\Gamma,U(m))$. To do so, we need to introduce the following definition.

\begin{definition}\label{def:goodpoly}
Let $\Gamma$ be a central extension as in (\ref{Gammak}) with non-zero $k$-invariant. Let $\bC[z]^m_{\Gamma}\subset\bC[z]$ be the set of all monic complex polynomials $p(z)$ of degree $m$ such that
\begin{enumerate}
\item all the roots of $p(z)$ are roots of unity;
\item If a root $\lambda$ of $p(z)$ is a primitive $k$-th root of unity, then the multiplicity of $\lambda$ in $p(z)$ is divisible by $\mu_k:=\prod_{i=1}^t k/(k,c_i)$, where $(k,c_i)$ denotes the greatest common divisor of $k$ and $c_i$.
\end{enumerate}
\end{definition}

\begin{corollary}\label{cor:noofcomprank1}
Suppose that $\Gamma$ is a central extension as in (\ref{Gammak}) and $\res:\Hom(\Gamma,U(m))\to \Hom(\bZ,U(m))\cong U(m)$ is the restriction map. If $\Gamma$ is abelian, then $\res$ is surjective. If $\Gamma$ is non-abelian, then the image of $\res$ is $\coprod_{p(z)\in\bC[z]^m_{\Gamma}\\}U(m)_{p(z)}$ and $\Hom(\Gamma,U(m))$ can be expressed as the union of its path-connected components as
\begin{equation}\label{homrank1disju}
\Hom(\Gamma,U(m))=\coprod_{p(z)\in\bC[z]^m_{\Gamma}}\text{Hom}(\Gamma,U(m))_{p(z)}.
\end{equation}
The number of path-connected components of $\Hom(\Gamma,U(m))$ is given by the coefficient of $x^m$ in the generating function
$
\prod_{k\geq 1}\frac{1}{(1-x^{\mu_k})^{\Phi(k)}}
$,
where $\mu_k$ is defined as in Definition \ref{def:goodpoly} and $\Phi(k)$ is Euler's
phi function.
\end{corollary}
\begin{proof}
If $\Gamma$ is abelian, then for any $X\in U(m)$, $(X,I_m,I_m,\ldots,I_m)\in \Hom(\Gamma,U(m))$ and so $\res$ is surjective. If $\Gamma$ is non-abelian, then for a degree $m$ complex polynomial $p(z)$, $\Hom(\Gamma,U(m))_{p(z)}$ is non-empty only if all the roots of $p(z)$ are roots of unity by Lemma \ref{Xeigen}. Suppose that $p(z)=\prod_{j=1}^s(z-\lambda_j)^{m_j}$ with distinct roots $\lambda_j$'s and each $\lambda_j$ is a primitive $k_j$-th root of unity. Let $q_j=\frac{1}{2\pi\sqrt{-1}}\log{\lambda_j}$ and $D_j=D_n(-c_1q_j,\ldots,-c_tq_j)$. Then by Theorem \ref{thmfib} and \ref{specAC}, $\Hom(\Gamma,U(m))_{p(z)}$ is non-empty if and only if each $B_n(U(m_j))_{D_j}$ is non-empty, which is true if and only if each $\sigma(D_j)=\prod_{i=1}^t|c_iq_j|=\prod_{i=1}^t k_j/(k_j,c_i)$ divides $m_j$, or equivalently, $p(z)\in\bC[z]^m_{\Gamma}$. Theorem \ref{thmfib} and Corollary \ref{cor:BnUgenD} also imply that for $p(z)\in\bC[z]^m_{\Gamma}$, $\Hom(\Gamma,U(m))_{p(z)}$ is path-connnected and $\res_{p(z)}$ is surjective. Since $\bC[z]^m_{\Gamma}$ is a finite discrete set in $\bC[z]$ and $\res,\chi:U(m)\to\bC[z]$ are continuous, we conclude that $\Hom(\Gamma,U(m))$ can be written as the disjoint union of path-connected components as in (\ref{homrank1disju}). Hence, the number of path-connected components of $\Hom(\Gamma,U(m))$ is equal to $|\bC[z]^m_{\Gamma}|$, which is obviously the coefficient of $x^m$ in $\prod_{k\geq 1}\frac{1}{(1-x^{\mu_k})^{\Phi(k)}}$ from Definition \ref{def:goodpoly}.
\end{proof}

\begin{example}\label{Heisenberg}
Let $\Gamma_1$ be the integral Heisenberg group, which can be described as the central extension $1\to\bZ\to \Gamma_1 \to \bZ^2\to 1$ with $k$-invariant $\omega=e_1^{\ast}\wedge e_2^{\ast}$. Thus $t=1,c_1=1$ and $\mu_k=k$. The generating function in Corollary \ref{cor:noofcomprank1} is given by
\begin{align*}
\prod_{k\geq 1}\frac{1}{(1-x^{\mu_k})^{\Phi(k)}}&=\frac{1}{1-x}\cdot\frac{1}{1-x^2}\cdot\frac{1}{(1-x^3)^2}\cdot\frac{1}{(1-x^4)^2}\cdot\frac{1}{(1-x^5)^4}\cdots\\
&=1+x+2x^2+4x^3+7x^4+13x^5+\text{higher terms}
\end{align*}
Hence the number of components of $\Hom(\Gamma_1,U(m))$ is $1,2,4,7,13$ for $m=1,\ldots,5$ respectively. We can also consider the generalized version of $\Gamma_1$ given by the central extension $1\to\bZ\to\Gamma_t\to\bZ^{2t}\to 1$ with $k$-invariant $\omega=\sum_{i=1}^te_i^{\ast}\wedge e_{t+i}^{\ast}$ where $t\geq 2$. In this case, $\mu_k=k^t$ and the generating function
\[\prod_{k\geq 1}\frac{1}{(1-x^{\mu_k})^{\Phi(k)}}%=\frac{1}{1-x}\cdot\frac{1}{1-x^{2^t}}\cdot\frac{1}{(1-x^{3^t})^2}\cdots
=(1+x+x^2+\ldots)(1+x^{2^t}+x^{2^{t+1}}+\ldots)(1+x^{3^t}+(x^{3^t})^2+\ldots)^2\cdots\]
From the coefficients it follows that $\Hom(\Gamma_t,U(m))$ is connected for $1\leq m\leq 2^t-1$, has two components for $2^t\leq m\leq 2^{t+1}-1$ and at least three components for $m\geq 2^{t+1}$.
\end{example}

\section{Rank $r>1$ case}
In this section we will study the space $\text{Hom}(\Gamma,U(m))$ for a central extension $\Gamma$ of the form (\ref{gammacentralext}) with rank $r>1$. We will decompose the restriction map $\Hom(\Gamma,U(m))\to\Hom(\bZ^r,U(m))$ into maps of fiber bundles over flag manifolds and relate the spaces to almost commuting tuples of unitary matrices. Results about the components, rational homology type and the associated representation space of $\text{Hom}(\Gamma,U(m))$ can then be obtained.
%***(may delete this paragraph later, didn't use Lambda explicitly)***
%For commuting unitary matrices $X_1,\ldots, X_r\in U(m)$, let $\lambda^1,\ldots,\lambda^s\in \bT^r$ be their distinct $r$-tuples of eigenvalues and $m_j=\dim{E_{\lambda^j}}$. Define $\Lambda:\Hom(\bZ^r,U(m))\to \text{Sym}_m\bT^r$ by
%$$ \Lambda(X_1,\ldots,X_r)=[\underbrace{\lambda^1,\ldots,\lambda^1}_{m_1 \text{ times}},\underbrace{\lambda^2,\ldots,\lambda^2}_{m_2 \text{ times}},\ldots,\underbrace{\lambda^s,\ldots,\lambda^s}_{m_s \text{ times}}].$$
%This map $\Lambda$ is an analogue of $U(m)\to\bC[z]$ which assigns a unitary matrix to its characteristic polynomial. Also, $\Lambda$ is $U(m)$-invariant and induces a homeomorphism $\Hom(\bZ^r,U(m))/U(m)\cong\text{Sym}_m\bT^r$.
We first determine the image of the restriction map $\Hom(\Gamma,U(m))\to\Hom(\bZ^r,U(m))$.

\begin{definition}
Given a central extension $\Gamma$ of the form (\ref{gammacentralext}) with $\omega_{ij}^l$ as in (\ref{gammalcentralext}), define $\omega:\bT^r\to\TnRZ$ by
\begin{equation}\label{omegafunction}
\omega(\lambda)_{ij}=\frac{1}{2\pi\sqrt{-1}}\log\left(\lambda_1^{\omega_{ij}^1}\lambda_2^{\omega_{ij}^2}\ldots \lambda_r^{\omega_{ij}^r}\right)
\end{equation}
for $1\leq i < j \leq n$. Let $\Hom_{\Gamma}(\bZ^r,U(m))\subset\Hom(\bZ^r,U(m))$ be the subset consisting of those commuting tuples $(X_1,X_2,\ldots,X_r)$ which have all their $r$-tuples of eigenvalues $\lambda\in \bT^r$ satisfying $\omega(\lambda)\in \TnQZ$ and $\sigma(\omega(\lambda))$ divides $\dim{E_{\lambda}}$.
\end{definition}

\begin{proposition}%\label{thm:imageofresgen}
The image of the restriction map $\res:\Hom(\Gamma,U(m))\to\Hom(\bZ^r,U(m))$ is
$\Hom_{\Gamma}(\bZ^r,U(m))$.
\end{proposition}
\begin{proof}
Suppose $(X_1,\ldots,X_r,A_1,\ldots,A_n)\in\Hom(\Gamma,U(m))$. Let $\lambda=(\lambda_1,\ldots,\lambda_r)\in \bT^r$ be a $r$-tuple of eigenvalues of $(X_1,\ldots,X_r)$. Since each pair of $A_i$ and $X_j$ commutes, $E_{\lambda}$ is invariant under each $A_i$. The restrictions of $A_i,A_j$ on $E_{\lambda}$ satisfy the relation
\begin{equation}\label{AiAjE}
[A_i|_{E_{\lambda}},A_j|_{E_{\lambda}}]
=(X_1^{\omega_{ij}^1}X_2^{\omega_{ij}^2}\ldots X_r^{\omega_{ij}^r})|_{E_{\lambda}}
=\lambda_1^{\omega_{ij}^1}\lambda_2^{\omega_{ij}^2}\ldots \lambda_r^{\omega_{ij}^r}Id_{E_{\lambda}}
=e^{2\omega(\lambda)_{ij}\pi\sqrt{-1}}Id_{E_{\lambda}}.
\end{equation}
By Lemma \ref{ACeigen} and Corollary \ref{cor:BnUgenD}, $\omega(\lambda)\in\TnQZ$ and $\sigma(\omega(\lambda))$ divides $\dim{E_{\lambda}}$. Hence $(X_1,\ldots,X_r)\in\Hom_{\Gamma}(\bZ^r,U(m))$.

On the other hand, suppose $(X_1,\ldots,X_r)\in\Hom_{\Gamma}(\bZ^r,U(m))$ is given. Using Corollary \ref{cor:BnUgenD}, one can find matrices $A_1,\ldots,A_n\in U(m)$ such that
$[A_i,A_j]=X_1^{\omega_{ij}^1}X_2^{\omega_{ij}^2}\ldots X_r^{\omega_{ij}^r}$
by constructing $A_1|_{E_{\lambda}},\ldots,A_n|_{E_{\lambda}}$ satisfying (\ref{AiAjE}) for each $r$-tuple of eigenvalues $\lambda\in \bT^r$ of $(X_1,\ldots,X_r)$ and taking direct sum. It shows that $(X_1,\ldots,X_r)$ lies in the image of $\res$.
\end{proof}

Let $\mathcal{F}_n$ be the free commutative monoid on $\TnQZ$. We can extend $\sigma:\TnQZ\to\bZ$ as defined in (\ref{def2sigma}) to a function on $\mathcal{F}_n$ by linearity. Therefore for $\mathcal{D}=\sum_{j=1}^sl_jD_j,l_j>0$, $\sigma(\mathcal{D})=\sum_{j=1}^sl_j\sigma(D_j)$. Let $\mathcal{F}_{n,m}$ be the preimage $\sigma^{-1}(m)$ of $m$ under $\sigma:\mathcal{F}_n\to\bZ$. Define $f:\Hom_{\Gamma}(\bZ^r,U(m))\to \mathcal{F}_{n,m}$ by
\[f(X_1,\ldots,X_r)=\sum_{j=1}^s \frac{\dim{E_{\lambda^j}}}{\sigma(\omega(\lambda^j))}\omega(\lambda^j),\]
where $\lambda^1,\ldots,\lambda^s$ are all the distinct $r$-tuples of eigenvalues of $(X_1,\ldots,X_r)$. For $\mathcal{D}\in \mathcal{F}_{n,m}$, let $\Hom(\bZ^r,U(m))_{\mathcal{D}}=f^{-1}(\mathcal{D})$, $\Hom(\Gamma,U(m))_{\mathcal{D}}=\res^{-1}(\Hom(\bZ^r,U(m))_{\mathcal{D}})$ and $\res_{\mathcal{D}}=\res|_{\Hom(\Gamma,U(m))_{\mathcal{D}}}$, where $\res:\Hom(\Gamma,U(m))\to\Hom_{\Gamma}(\bZ^r,U(m))$ is the restriction map. Since $f$ is continuous if $\mathcal{F}_{n,m}$ is given the discrete topology, we obtain the following theorem.

\begin{theorem}\label{HomD}
For a central extension $\Gamma$ of the form (\ref{gammacentralext}), $\Hom_{\Gamma}(\bZ^r,U(m))$, $\Hom(\Gamma,U(m))$ and $\res$ can be expressed as disjoint unions indexed by $\mathcal{F}$ and there is a commutative diagram
\[
\xymatrix{
\Hom(\Gamma,U(m)) \ar@{=}[r]\ar[d]^{\res}& \coprod_{\mathcal{D}\in\mathcal{F}_{n,m}}\text{Hom}(\Gamma,U(m))_{\mathcal{D}}\ar[d]^{\coprod_{\mathcal{D}\in\mathcal{F}_{n,m}}\res_{\mathcal{D}}}\\
\Hom_{\Gamma}(\bZ^r,U(m)) \ar@{=}[r] &\coprod_{\mathcal{D}\in\mathcal{F}_{n,m}}\text{Hom}(\bZ^r,U(m))_{\mathcal{D}}.
}
\]
\end{theorem}

Because of Theorem \ref{HomD}, we can focus on $\Hom(\Gamma,U(m))_{\mathcal{D}}$ for each $\mathcal{D}\in\mathcal{F}_{n,m}$. The next two theorems break down these spaces and relate them to $B_n(U(m))$.

\begin{theorem}%\label{homfiber}
Suppose that $\mathcal{D}=\sum_{j=1}^sl_jD_j\in \mathcal{F}_{n,m}$ with $l_j>0$ and $D_i\neq D_j$ for $i\neq j$. Let $m_j=l_j\sigma(D_j)$. Then there are $U(m)$-equivariant homeomorphisms
\[
\Hom(\Gamma,U(m))_{\mathcal{D}}\cong U(m)\times_{\prod_{j=1}^sU(m_j)}\left(\prod_{j=1}^s\Hom(\Gamma,U(m_j))_{l_jD_j}\right),
\]
and
\[
\Hom(\bZ^r,U(m))_{\mathcal{D}}\cong U(m)\times_{\prod_{j=1}^sU(m_j)}\left(\prod_{j=1}^s\Hom(\bZ^r,U(m_j))_{l_jD_j}\right).
\]
Hence, both $\Hom(\Gamma,U(m))_{\mathcal{D}}$ and $\Hom(\bZ^r,U(m))_{\mathcal{D}}$ are fiber bundles over
%the flag manifold
$U(m)/(\prod_{j=1}^sU(m_j))$. The map $\res_{\mathcal{D}}$ is a map of fiber bundles which fits into the following commutative diagram.
\[
\xymatrixcolsep{6pc}
\xymatrix{
\prod_{j=1}^s\Hom(\Gamma,U(m_j))_{l_jD_j}\ar[d]\ar[r]^{\prod_{j=1}^s\res_{l_jD_j}}&\prod_{j=1}^s\Hom(\bZ^r,U(m_j))_{l_jD_j}\ar[d]\\
\Hom(\Gamma,U(m))_{\mathcal{D}}\ar[d]\ar[r]^{\res_{\mathcal{D}}}&\Hom(\bZ^r,U(m))_{\mathcal{D}}\ar[d]\\
\bigslant{U(m)}{\left(\prod_{j=1}^sU(m_j)\right)}\ar@{=}[r] & \bigslant{U(m)}{\left(\prod_{j=1}^sU(m_j)\right).}
}
\]
\end{theorem}
\begin{proof}
The space $U(m)/(\prod_{j=1}^sU(m_j))$ can be considered as a flag manifold with points represented by ordered tuples $(V_1,\ldots,V_s)$, where each $V_j\subset \bC^m$ is a $m_j$-dimensional complex subspace and $\bC^m=\oplus_{j=1}^sV_j$ is an orthogonal decomposition of $\bC^m$. Let \[g:\Hom(\bZ^r,U(m))_{\mathcal{D}}\to U(m)/(\textstyle\prod_{j=1}^sU(m_j))\]
be defined by
$g(X_1,\ldots,X_r)=(V_1,\ldots,V_s)$
with $V_j=\oplus_{\lambda\in\omega^{-1}(D_j)}E_{\lambda}$. Then $g,g\circ\res_{\mathcal{D}}$ are continuous and $U(m)$-equivariant. Since the $U(m)$-action on $U(m)/(\prod_{j=1}^sU(m_j))$ is transitive, we obtain the $U(m)$-homeomorphisms in the theorem by Lemma \ref{lem:fiberbundle}. It is obvious that $\res_{\mathcal{D}}$ can be realized as a map of fiber bundles induced by $\prod_{j=1}^s\res_{l_jD_j}$ and hence the given diagram commutes.
\end{proof}

Recall that we defined $B_D=B_n(U(\sigma(D)))$ for $D\in\TnQZ$. In Corollary \ref{cor:BnDq} we showed that it can be used as the building blocks for approximating $B(U(m))_D$ rationally. We can do it for $\Hom(\Gamma,U(m))_{lD}$ too.

\begin{theorem}\label{lDQ}
For $\mathcal{D}=lD$ and $m=\sigma(\mathcal{D})=l\sigma(D)$, there is a commutative diagram
\[
\xymatrix{
\bigslant{U(m)\times_{U(\sigma(D))^l}(B_D\times \omega^{-1}(D))^l}{\Sigma_l}\ar[r]^(.65){\phi_{\mathcal{D}}}\ar[d] &\Hom(\Gamma,U(m))_{\mathcal{D}}\ar[d]^{\res_{\mathcal{D}}} \\
\bigslant{U(m)\times_{U(\sigma(D))^l}\omega^{-1}(D)^l}{\Sigma_l}\ar[r]^(.58){\eta_{\mathcal{D}}} &\Hom(\bZ^r,U(m))_{\mathcal{D}}.
}
\]
where the horizontal maps are rational homology equivalences for any $l\geq 1$ and homeomorphisms for $l=1$. Here $\omega:\bT^r\to\TnRZ$ is the function as defined in (\ref{omegafunction}).
\end{theorem}
\begin{remark}
The horizontal maps are defined in the same way as $\phi_D$ in (\ref{phigeneralD}). The domain of $\eta_{\mathcal{D}}$ can be considered as the space of unordered $l$-tuples of pairwisely orthogonal $\sigma(D)$-dimensional complex subspaces of $\bC^m$ with a label $\lambda\in\omega^{-1}(D)$ attached to each of the subspaces. A commuting tuple $(X_1,\ldots,X_r)\in\Hom(\bZ^r,U(m))_{\mathcal{D}}$ can be uniquely determined by specifying these $\lambda$ as the $r$-tuples of eigenvalues of the subspaces they label. This defines $\eta_{\mathcal{D}}$. The map $\phi_{\mathcal{D}}$ is defined similarly: Each point of its domain carries the additional data of a $D$-commuting tuple of unitary automorphisms of each of those complex subspaces. The additional data is used for constructing the matrices $A_i$ in its image $(X_1,\ldots,X_r,A_1,\ldots,A_n)\in\Hom(\Gamma,U(m))_{\mathcal{D}}$ under $\phi_{\mathcal{D}}$.
\end{remark}
\begin{proof}
It is clear that the diagram commutes. Suppose $(X_1,\ldots,X_r)\in\Hom(\bZ^r,U(m))_{\mathcal{D}}$. Let $\lambda^1,\ldots,\lambda^s\in\bT^r$ be all its distinct $r$-tuples of eigenvalues. Note that each $\lambda^j\in\omega^{-1}(D)$ and $\dim E_{\lambda^j}=m_j=l_j\sigma(D)$ for some $l_j\in\bZ$. It is easy to see that \[\eta_{\mathcal{D}}^{-1}(X_1,\ldots,X_r)\cong \prod_{j=1}^s \bigslant{U(m_j)/U(\sigma(D))^{l_j}}{\Sigma_{l_j}},\]
which is $\bQ$-acyclic in general and is a point if $l=1$. Thus $\eta_{\mathcal{D}}$ satisfies the properties stated in the theorem. For $\phi_\mathcal{D}$, it is clear that there exist homeomorphisms
\[\res_{\mathcal{D}}^{-1}(X_1,\ldots,X_r)\cong\prod_{j=1}^sB_n(U(m_j))_D\]
\[(\res_{\mathcal{D}}\circ\phi_{\mathcal{D}})^{-1}(X_1,\ldots,X_r)\cong\prod_{j=1}^s\bigslant{U(m_j)\times_{U(\sigma(D))^{l_j}}B_D\,^{l_j}}{\Sigma_{l_j}}\]
such that the restriction of $\phi_{\mathcal{D}}$ on $(\res_{\mathcal{D}}\circ\phi_{\mathcal{D}})^{-1}(X_1,\ldots,X_r)$ can be regarded as the product of
\[\phi_D:\bigslant{U(m_j)\times_{U(\sigma(D))^{l_j}}B_D\,^{l_j}}{\Sigma_{l_j}}\to B_n(U(m_j))_D\]
under these homeomorphisms. It follows from Theorem \ref{thm:BnD} that $\phi_\mathcal{D}$ is a rational homology equivalence for any $l\geq 1$ and a homeomorphism for $l=1$.
\end{proof}

The space $\omega^{-1}(D)$ can be studied using linear systems over $\bR/\bZ$.

\begin{definition}\label{BC}
Suppose that $\Gamma$ is a central extension of the form (\ref{gammacentralext}) and $\omega_{ij}^l$ are the coefficients of its $k$-invariant as in (\ref{gammalcentralext}).
\begin{enumerate}
\item Let $\Omega$ be the $C^n_2\times r$ matrix with rows indexed by $(i,j),1\leq i<j\leq n$, columns indexed by $1\leq l\leq r$ and the $l$-th entry on the $(i,j)$-th row equals to $\omega_{ij}^l$.

\item Let $Q$ be a row echelon form of $\Omega$ over $\bZ$. Define $B$ to be the absolute value of the product of the pivot entries of $Q$. The number $B$ is independent of the choice of $Q$.

\item Let $R$ be the reduced row echelon form of $\Omega$ over $\bQ$. The columns of $R$ can be regarded as vectors in $\bQ^{C^n_2}$. Let $C(R)$ be the $\bZ$-submodule of $(\bQ/\bZ)^{C^n_2}$ generated by the images of the columns of $R$ under the quotient map $\bQ^{C^n_2}\to(\bQ/\bZ)^{C^n_2}$. Define $C$ to be the number of elements in $C(R)$.

\item Define $P(\Omega)=B/C.$
\end{enumerate}
\end{definition}

\begin{example}
Let $r=4,n=3$,
\[Q=
\begin{pmatrix}
4 &-1&1&-6\\
0 &3&1&2\\
0&0&0&0
\end{pmatrix}
\text{ and so }
R=
\begin{pmatrix}
1 &0&1/3&-4/3\\
0 &1&1/3&2/3\\
0&0&0&0
\end{pmatrix}.
\]
Then $B=(4)(3)=12$ and $C(R)$ is the $\bZ$-submodule of $(\bQ/\bZ)^4$ generated by
\[
\begin{bmatrix}
1\\0\\0
\end{bmatrix}
\equiv
\begin{bmatrix}
0\\1\\0
\end{bmatrix}
\equiv
\begin{bmatrix}
0\\0\\0
\end{bmatrix},
\begin{bmatrix}
1/3\\1/3\\0
\end{bmatrix}
\text{ and }
\begin{bmatrix}
-4/3\\2/3\\0
\end{bmatrix}
\equiv 2
\begin{bmatrix}
1/3\\1/3\\0
\end{bmatrix}.\]
Hence $C=3$ and $P(\Omega)=12/3=4$.
\end{example}

\begin{lemma}\label{lem:omegainvD}
Let $\omega:\bT^r\to\TnRZ$ be the map (\ref{omegafunction}) and $\Omega,P(\Omega)$ be as in Definition \ref{BC}. Then for any $D\in \TnQZ$,
$\omega^{-1}(D)$ is either empty or homeomorphic to a disjoint union of $P(\Omega)$ copies of $\bT^{\text{nul}(\Omega)}$, where$\text{ nul}(\Omega)$ is the nullity of $\Omega$.
\end{lemma}
\begin{proof}
Suppose $\omega^{-1}(D)$ is non-empty. A $r$-tuple $\lambda=(\lambda_1,\ldots,\lambda_r)\in\omega^{-1}(D)$ if and only if
\[\sum_{k=1}^r\omega_{ij}^k\frac{\log\lambda_k}{2\pi\sqrt{-1}}=d_{ij}\]
for all $1\leq i<j\leq n$. Let $D'\in(\bQ/\bZ)^{C^n_2}$, with entries indexed by $(i,j),1\leq i < j\leq n$, be obtained from $D$ by rewriting the entries in a column vector. Suppose $Q,R,B,C$ are as in Definition \ref{BC} and   $x_k=\frac{\log\lambda_k}{2\pi\sqrt{-1}}$. Then $\omega^{-1}(D)$ is homeomorphic to the solution space $S$ of the equivalent linear systems
$\Omega \vec{x} = D' \Longleftrightarrow Q \vec{x} = D''$
over $\bQ/\bZ$. Here $D''$ is obtained from $D'$ by performing the same elementary operations used to obtain $Q$ from $\Omega$. Without loss of generality, assume $x_1,\ldots,x_p$ are the basic variables and $x_{p+1},\ldots,x_r$ are the free variables in the system $Q \vec{x} = D''$. For $1\leq s\leq p$, let $b_s$ be the pivot entry in the $s$-th column of $Q$. Then for points in $S$, $b_1x_1,\ldots,b_px_p\in\bR/\bZ$ are uniquely determined by any $(x_{p+1},\ldots x_r)\in T:=(\bR/\bZ)^{r-p}\cong\bT^{r-p}$. Hence, $S$ is a $B$-sheeted covering space over $T$. Note that all entries below the $p$-th row of $R$ are zero. Let $\vec{v_s}\in\bR^p$ be the first $p$ entries of the $s$-th column of $R$ and $\vec{w_s}$ be its image under the quotient map $\bR^p\to(\bR/\bZ)^p$. Then a loop in $T$ parametrized by $x_s\in\bS^1$ with other coordinates fixed lifts to a path connecting $\vec{x}$ and $\vec{x}+\begin{bmatrix}\vec{w_s}\\ \vec{0}\end{bmatrix}$ in $S$. From this it can be deduced that $S$, and hence $\omega^{-1}(D)$, is homeomorphic to a disjoint union of $B/C=P(\Omega)$ copies of $\bT^{r-p}$, where $r-p=\text{ nul}(\Omega)$ is the nullity of $\Omega$. This proves the lemma.
\end{proof}

Combining Theorems \ref{thm:BnD}, \ref{lDQ} and Lemma \ref{lem:omegainvD}, we obtain the following result about the number of components, rational cohomology and the associated representation space for $\Hom(\Gamma,U(m))_{\mathcal{D}}$.

\begin{theorem}
Suppose that $\mathcal{D}=\sum_{j=1}^sl_jD_j\in \mathcal{F}_{n,m}$ with $l_j>0$ and $D_i\neq D_j$ for $i\neq j$. Let $l=\sum_{j=1}^sl_j$ and $m_j=l_j\sigma(D_j)$. Then
$\Hom(\Gamma,U(m))_{\mathcal{D}}$ is non-empty if and only if $\omega^{-1}(D_j)$ is non-empty for all $j=1,\ldots,s$. In that case, $\Hom(\Gamma,U(m))_{\mathcal{D}}$ has
\[
\prod_{j=1}^s {P(\Omega)+l_j-1 \choose l_j}
\]
components and there is a rational homology equivalence

\begin{equation*}%\label{HomDQ}
\bigslant {\left[\bigslant{\left(U(m)/\bT^l\right)\times (\coprod_{P(\Omega)}\bT^{n+\text{nul}(\Omega)})^l}{\bZ_{\mathcal{D}}}\right]}{\prod_{j=1}^s\Sigma_{l_j}}
\to\text{Hom}(\Gamma,U(m))_\mathcal{D},
\end{equation*}
where the action of the finite abelian group $\bZ_{\mathcal{D}}:= \prod_{j=1}^s(\bZ_{D_j})^{l_j}$ on the space
$\left(U(m)/\bT^l\right)\times (\coprod_{P(\Omega)}\bT^{n+\text{nul}(\Omega)})^l$ is trivial on rational cohomology. Moreover the map induces a homeomorphism
%\bigslant{\Hom(\Gamma,U(m))_{\mathcal{D}}}{U(m)}\cong \prod_{j=1}^s\bigslant{\left((\bS^1)^n\times \omega^{-1}(D_j)\right)^{l_j}}{\Sigma_{l_j}}\]
\begin{equation*}%\label{RepDhomeo}
\bigslant{\Hom(\Gamma,U(m))_{\mathcal{D}}}{U(m)}\cong \prod_{j=1}^s\left[\bigslant{\left(\coprod_{P(\Omega)}\bT^{n+\text{nul}(\Omega)}\right)^{l_j}}{\Sigma_{l_j}}\right]
\end{equation*}
after passing to quotients by the action of $U(m)$.
%Hence each component of the representation space $\Rep(\Gamma,U(m))$ is a product of symmetric products of disjoint unions of tori
%of the form $\bT^{n+\text{nul}(\Omega)}$.
\end{theorem}

Finally we note that the rational cohomology of the spaces $B_n(U(m))$ and $\Hom(\Gamma, U(m))$ can
be computed using standard spectral sequence arguments and invariant theory.
%For instance, if $D\in\TnQZ$ and $m=l\sigma(D)$. Then
%$$
%H^{\ast}(B_n(U(m))_D,\bQ)=\left(\bigslant{E[d_{2t-1},y_{ik}]\otimes P[x_k]}{S_i(x)}\right)^{\Sigma_l},
%$$
%where $E$ and $P$ is the exterior algebra and polynomial algebras with generators indexed by $1\leq i\leq n,1\leq k\leq l,l+1\leq t\leq m$. The degrees of the generators are $|d_{2t-1}|=2t-1,|x_{k}|=2$ and $|y_{ik}|=1$. The ideal $S_i(x)$ is generated by the symmetric polynomials in $x_1,\ldots,x_l$. The symmetric group $\Sigma_l$ acts on the generators of $\Lambda[d_{2t-1},x_k,y_{ik}]$ by permuting the subscripts $k$ of $x_k,y_{jk}$ and fixing $d_{2t-1}$.
Details are left to the
interested (and highly motivated) reader.

%\begin{proof}
%$$U(m)\times_{U(\sigma(D))\wr \Sigma_l}(B_D\times \omega^{-1}(D))^l$$
%
%$$\Hom(\Gamma,U(m))_{\mathcal{D}}\cong %U(m)\times_{\left(\prod_{j=1}^sU(m_j)\right)}\left(\prod_{j=1}^s\Hom(\Gamma,U(m_j))_{l_jD_j}\right)$$

%\end{proof}

%\begin{example}
%In the trivial case where the $k$-invariant is zero, we have $\Gamma\cong \bZ^{n+r}$, $B=C=P(\Omega)=1$ and %$\text{nul}(\Omega)=r$. Also, $\omega^{-1}(D)$ is empty if and only if $D$ is not the zero matrix $\bf{O}$. Therefore, %$\Hom(\Gamma,U(m))_{\mathcal{D}}$ is non-empty if and only if $\mathcal{D}=m\bf{O}$. Then (\ref{HomDQ}) and %(\ref{RepDhomeo}) reduce to the well-known results (cite)
%$$
%\bigslant{\left(U(m)/\bT^m\right)\times\bT^{n+r}}{\Sigma_m}\to \text{Hom}(\bZ^{n+r},U(m))
%$$
%and
%\begin{align*}
%\text{Hom}(\bZ^{n+r},U(m))/U(m)&= \Hom(\Gamma,U(m))/U(m)\\
%&= \Hom(\Gamma,U(m))_{m\bf{O}}/U(m)\\
%&\cong \bigslant{(\bT^{n+r})^m}{\Sigma_m}.
%\end{align*}

%\end{example}

\bibliographystyle{plain} % bibliography style

\end{document}